\documentclass[11pt]{amsart}

\usepackage{amsmath, amscd, amssymb}
\usepackage{graphicx, psfrag}
\usepackage[frame,cmtip,arrow,matrix,line,graph,curve]{xy}
\usepackage{graphpap, color}
\usepackage[mathscr]{eucal}
\usepackage{cancel}
\usepackage{verbatim}
 
\numberwithin{equation}{section}

\def\sO{{\mathscr O}}

\newcommand{\CC}{\mathbb{C}}

\newcommand{\PP}{\mathbb{P}}
\newcommand{\QQ}{\mathbb{Q}}
\newcommand{\RR}{\mathbb{R}}
\newcommand{\ZZ}{\mathbb{Z}}

\newcommand{\cal}{\mathcal}

\def\cC{{\cal C}}

\def\cF{{\cal F}}

\def\cH{{\cal H}}

\def\cL{{\cal L}}

\def\loc{\mathrm{loc}}

\def\and{\quad{\rm and}\quad}
\def\lra{\longrightarrow }
\def\mapright#1{\,\smash{\mathop{\lra}\limits^{#1}}\,}

\newtheorem{prop}{Proposition}[section]
\newtheorem{theo}[prop]{Theorem}
\newtheorem{lemm}[prop]{Lemma}
\newtheorem{coro}[prop]{Corollary}
\newtheorem{rema}[prop]{Remark}
\newtheorem{exam}[prop]{Example}
\newtheorem{ques}[prop]{Question}
\newtheorem{defi}[prop]{Definition}

\newtheorem{assu}[prop]{Assumption}

\def\beq{\begin{equation}}
\def\eeq{\end{equation}}

\def\virt{^{\mathrm{vir}} }
\def\virtloc{\virt_\loc}

\def\DM{Deligne-Mumford }

\def\zero{\mathrm{zero} }

\def\bl{\bigl(}
\def\br{\bigr)}
\def\redd{{\mathrm{red}}}

\def\fX{\mathfrak{X} }

\def\tE{\widetilde{E} }

\def\tX{\widetilde{X} }

\def\bp{\mathbf{p}} 
\def\bq{\mathbf{q}}

\def\tX{\widetilde{X} }

\def\tS{\widetilde{S}}
\def\tE{\widetilde{E}}
\def\cC{\mathcal{C}}
\def\zero{\mathrm{zero}}

\def\tsig{\tilde{\sigma}}
\def\tY{\widetilde{Y}}
\def\DM{Deligne-Mumford }
\def\ih{I\!H}
\def\ulambda{{\lambda}}
\def\ulog{^{\mathrm{log}}}
\def\urig{^{\mathrm{rig}}}
\def\ugamma{{\underline{\gamma}}}
\def\rmgn{\overline{M}_{g,n}}
\def\git{/\!/}

\title{Quantum singularity theory via cosection localization}
\date{2018.10.31.}
\author{Young-Hoon Kiem}
\address{Department of Mathematics and Research Institute
of Mathematics, Seoul National University, Seoul 08826, Korea}
\email{kiem@snu.ac.kr}
\author{Jun Li}
\address{Department of Mathematics,  Stanford University, CA 94305, USA}
\email{jli@math.stanford.edu}

\thanks{YHK was partially supported by Samsung Science and Technology Foundation SSTF-BA1601-01; JL was paritally supported by NSF grants DMS-1564500 and DMS-1601211.}
\keywords{FJRW invariant, cohomological field theory, cosection localization.}

\begin{document}
\begin{abstract} We generalize the cosection localized Gysin map to intersection homology and
Borel-Moore homology, which  provides us with a purely topological construction of the 
Fan-Jarvis-Ruan-Witten invariants and some GLSM invariants.\end{abstract}
\maketitle

\section{Introduction}\label{S1}

In this paper, we generalize the cosection localized Gysin map in \cite{KLc} to intersection homology 
and Borel-Moore homology, which  provides us with  a purely topological construction of the cohomological field theory of Fan-Jarvis-Ruan-Witten (FJRW for short) invariants in \cite{FJR} and some gauged linear sigma model (GLSM for short) invariants in \cite{FJRgl, CFGK}.

Gromov-Witten (GW for short) invariants in algebraic geometry enumerate curves in smooth projective varieties and have been the topic of intensive research for decades. 
The computation of GW invariants is known to be very hard and is still far from complete even for the most desired examples like the Calabi-Yau (CY for short) 3-fold in the projective space $\PP^4$ over $\CC$ defined by  the Fermat quintic $\sum_{i=1}^5z_i^5.$

Given a nondegenerate (quasi-)homogeneous polynomial $w$ on $\CC^N$, Witten in \cite{Witt} introduced an equation for smooth sections of spin bundles over curves and conjectured that the intersection numbers on the solution space should be equivalent to the GW invariants of the projective hypersurface defined by $w$. Rigorous mathematical theories for the new invariants, called the FJRW invariants, were provided by Fan-Jarvis-Ruan in \cite{FJR} through analysis and Polishchuk-Vaintrob in \cite{PV} through the categories of matrix factorizations. 
A purely algebraic theory through virtual cycles was provided for the narrow case in \cite{CLL} via the cosection localization principle of virtual cycles in \cite{KLc}. 
The FJRW invariants were generalized to the setting of GLSM in \cite{FJRgl} where the GLSM invariants are defined, for the narrow case only, by applying the cosection localization in \cite{KLc} along the line of arguments in \cite{CLp, CLL}.  
However for the general setting where broad markings are allowed, the insertion classes are real classes, often of odd real dimension. As the cosection localization developed in \cite{KLc} is entirely algebraic, it was difficult to expect that the cosection localization may provide an algebraic theory for the broad sectors as well. 

The goal of this paper is to show that it is indeed possible to construct all the FJRW invariants, as well as some GLSM invariants, by generalizing the cosection localization to real classes.
More precisely, we define the cosection localized Gysin maps for intersection homology (Theorem \ref{b18}) and Borel-Moore homology  (Theorem \ref{f18}). 
Then we construct cohomological field theories of the singularities $w^{-1}(0)$ (resp. the LG models $\CC^N\git G\mapright{w}\CC$), based on 
the state space given by intersection homology (resp. Borel-Moore homology)
which contains the state space in the FJRW theory (resp. GLSM theory). 
The homogeneity issue of \cite{PV} is trivial for us and our construction is purely algebraic and topological, in line with the constructions in \cite{KLc, KLk, Kspo}.

\bigskip

In order to explain the background and motivation, we need a more precise setup. 
Let $w:\CC^N\to \CC$ be a nondegenerate quasi-homogeneous polynomial (cf. \S\ref{S4.2}) like the Fermat quintic $\sum_{i=1}^5z_i^5.$
Let $\hat{G}$ be a subgroup of $(\CC^*)^N$ which makes $w$ semi-invariant, i.e. there is a surjective  homomorphism $\chi:\hat{G}\to \CC^*$ such that $w(g\cdot x)=\chi(g)\,w(x)$. The kernel of $\chi$ is a finite group $G$ and the pair $(w,G)$ is the input datum for the Landau-Ginzburg model of the singularity $w^{-1}(0)$. 
Consider the quotient stack 
\[ \fX=[(\CC^N\times \CC)/\hat{G} ]\]
where $\hat{G}$ acts on the first factor $\CC^N$ by the inclusion $\hat{G}\to (\CC^*)^N$ 
and on the second factor $\CC$ by the character $\chi^{-1}$. 
Notice that $w$ induces a $\hat{G}$-invariant function $\mathbf{w}(x,t)=t\cdot w(x)$ on $\CC^N\times \CC$. 
The stack $\fX$ admits two open substacks which are GIT quotients: 
By deleting the zero in the first factor, we get an (orbi-)line bundle $\fX_+$ on the projective orbifold $\CC^N-\{0\}/\hat{G}=:\PP^{N-1}_G$, which comes equipped with a function $\mathbf{w}|_{\fX_+}$ whose critical locus is the hypersurface $\PP w^{-1}(0)$ in $\PP^{N-1}_G$ defined by $w$. 
By deleting the zero in the second factor, we have the Landau-Ginzburg (LG for short) model 
$$\fX_-=[\CC^N/G\mapright{w} \CC]. $$ 
Witten's idea is that the curve counting (i.e. GW invariants) in the projective model $\PP w^{-1}(0)$ should be computable by the LG  model $w:\CC^N/G\to \CC$. 

Of course, the first step in this program, often called the LG/CY correspondence, of comparing curve countings in the  two models $\fX_\pm$ of the stack $\fX$, should be developing a curve counting theory for the LG model $(w,G)$. 
Following Witten's ideas in \cite{Witt}, Fan, Jarvis and Ruan studied intersection theory 
on the space of solutions to Witten's equation
for sections of spin bundles $\{L_i\}_{1\le i\le N}$ on stable pointed twisted curves $(C,p_1,\cdots, p_n)$. 
Through analysis, they defined the FJRW invariants which were proved to satisfy nice axioms
like the splitting principle which are codified
as a cohomological field theory \cite[\S4]{FJR}. (See \S\ref{S4.3} for the definition of a cohomological field theory.)
The term \emph{quantum singularity theory} in the title refers to a cohomological field theory from a singularity. 

Slightly later, Polishchuk and Vaintrob in \cite{PV} provided an alternative approach for the FJRW invariants. 
They constructed a universal matrix factorization on the moduli space of spin curves with sections
and then obtained a cohomological field theory by a Fourier-Mukai type transformation for matrix factorizations. 
This elegant theory however lacks in concrete interpretation in terms of cycles as in \cite{FJR} and hence some basic properties like the homogeneity of dimension are not so obvious but required another long paper.
Such a cycle theory in algebraic geometry was provided by  H.-L. Chang, J. Li and W.-P. Li   
in \cite{CLL}, where the authors proved that the FJRW invariants for narrow sectors are 
in fact integrals on the cosection localized virtual cycle of the moduli space $X=\{(C,p_j,L_i,x_i)\}$ 
of spin curves $(C,p_j,L_i)$ with sections $x_i\in H^0(L_i)$.
Here the narrow sector means that the stabilizer of each marking $p_j$ has only one fixed point $0$ in the fiber
$\oplus_{i=1}^N L_i|_{p_j}$ of the spin bundles. Otherwise, we say that the marking is broad. 
In \cite{FJRgl}, by applying the cosection localized virtual cycle construction in \cite{CLp, CLL} to the 
more general GLSM setting where $G$ is not necessarily finite (cf. \S\ref{S5}), Fan, Jarvis and Ruan defined the GLSM invariants
for narrow sectors. Quite recently, by generalizing the matrix factorization construction of Polishchuk and Vaintrob in \cite{PV},
Ciocan-Fontanine, Favero, Guere, Kim and Shoemaker in \cite{CFGK} defined all (broad and narrow) GLSM invariants 
for convex hybrid GLSMs.\footnote{We thank Bumsig Kim for drawing our attention to \cite{CFGK} after completing the first version of this paper.}

In the presence of broad markings, the cosection localization in \cite{KLc} does not work because the cosection of the obstruction sheaf is not well defined everywhere. 
Moreover, the insertions for broad markings are not algebraic, often odd real dimensional, and thus one is led to believe that cosection localization may not provide us with the entire FJRW invariants or GLSM invariants. 
In this paper, we generalize the cosection localized Gysin map to topological cycles, 
and construct quantum singularity theories as well as some GLSM theories,
as a direct application of the cosection localization. 

\bigskip

Let $Y$ be a \DM stack and $E$ be a vector bundle of rank $r$ equipped with a section $s$ whose zero locus is $X$.
Then the cap product with the Euler class $e(E,s)\in H^{2r}(Y,Y-X)$ gives us the Gysin map
$$ s^!:H_i(Y)\lra H_{i-2r}(X),\quad s^!(\xi)=\xi\cap e(E,s) $$
where $H_i(\cdot)$ denotes the Borel-Moore homology. For the FJRW (or GLSM in general) theory however, this is not enough
since the moduli space $X$ is not compact and so we cannot integrate cohomology classes over $X$ to obtain invariants. 
On the other hand, in the FJRW theory, we have a homomorphism (cosection) $\sigma:E\to \sO_Y$ 
satisfying $\sigma\circ s=0$ and the common zero locus of $s$ and $\sigma$ is the moduli space $S$
of (rigidified) stable spin curves which is certainly compact.  
In this paper, we prove that the images by $s^!$ of intersection homology classes 
are in fact supported on $S$. In other words, 
we construct a homomorphism (Theorem \ref{b18})
$$ s^!_\sigma:\ih_i(Y)\lra H_{i-2r}(S) $$
by generalizing the arguments in \cite{KLc} and using the decomposition theorem in \cite{BBD},
where $\ih_*(\cdot)$ denotes the middle perversity intersection homology.
The restriction to the broad markings together with rigidification gives us a smooth morphism
$q:Y\to Z$ where the intersection homology of $Z$ contains the $n$th power of the state space 
$\cH=\oplus_{\gamma\in G}\ih_{N_\gamma}(w_\gamma^{-1}(0))^G$ for the FJRW theory. 
Composing the pullback 
$$\cH^{\otimes n}\hookrightarrow \ih_*(Z)\mapright{q^*} \ih_*(Y) $$
with $s^!_\sigma$ and the pushforward $\frac1{\deg\,\mathrm{st}} \mathrm{st}_*$ for the forgetful map $\mathrm{st}:S\to \rmgn$, 
we obtain homomorphisms
$$ \Omega_{g,n}=\frac{(-1)^D}{\deg\,\mathrm{st}} \mathrm{st}_*\circ s^!_\sigma\circ q^*:\cH^{\otimes n}\lra H_*(\rmgn)\cong H^*(\rmgn)$$
which we prove to satisfy all the axioms for cohomological field theories (Theorem \ref{a19}). Here $D$ is minus the Riemann-Roch number of the spin bundles.
The construction for GLSMs is similar with $\ih$ replaced by Borel-Moore homology  (cf. Theorem \ref{f18} and \S\ref{S5}). 

We expect that the quantum singularity theory via cosection localization constructed in this paper does coincide with the cohomological field theory of Fan-Jarvis-Ruan in \cite{FJR}. 
For narrow sectors, all constructions give us the same invariants as proved in \cite{CLL}. 
However, as we are unable to follow the details of the construction 
in \cite{FJR}, the equivalence of the two constructions for broad sectors is still open.

\bigskip

In \cite{CKLa}, we construct an algebraic virtual cycle $[\tilde{X}]\virt_{\mathrm{loc}}$ after blowing up 
$w^{-1}(0)$ at the origin and prove that the convolution product (i.e. a Fourier-Mukai type transform) with kernel $[\tilde{X}]\virt_{\mathrm{loc}}$ coincides with the cohomological field theories defined in this paper. 

\bigskip

The layout of this paper is as follows: In \S\ref{S2}, we recall and establish useful properties about Borel-Moore homology and intersection homology. 
In \S\ref{S3}, we construct the cosection localized Gysin maps for intersection homology and Borel-Moore homology. 
In \S\ref{S4}, we construct quantum singularity theories by applying the cosection localized Gysin maps.
In \S\ref{S5}, we provide a topological construction of some GLSM invariants.

\bigskip

All varieties, schemes and stacks are separated and defined over $\CC$ in this paper. We will use only the classical topology of algebraic varieties and schemes.
All the topological spaces in this paper are locally compact Hausdorff countable CW complexes.  
All the cohomology groups in this paper have complex coefficients although most of the arugments work for rational coefficients. 

We thank Huai-Liang Chang for extended discussions that provided impetus to  developing the current theory.  
We also thank Jinwon Choi, Wei-Ping Li and Yongbin Ruan for useful discussions.

\bigskip

\section{Borel-Moore homology and intersection homology}\label{S2}
In this section, we recall and establish some useful properties of the Borel-Moore homology and intersection homology from \cite{Iver, Bred, FulY, BBD, Borel, GM83, GM85}.

\subsection{Borel-Moore homology}\label{S2.1}
Let $X$ be a locally compact Hausdorff countable CW complex. 
Let $\cC^p_X$ denote the sheafification of the presheaf assigning
$$U\mapsto C^p(U)=\{\text{singular }p\text{-cochains on }U \}.$$
Then the complex
\beq\label{b10} \cC_X^\cdot: 0\lra \cC_X^0\lra \cC_X^1\lra \cC_X^2\lra \cdots \eeq
of singular cochains is a flabby resolution of the constant sheaf $\CC$.
The Verdier dual $\mathbb{D}_X(\cC_X^\cdot)$ of $\cC_X^\cdot$ is defined as the complex of flabby sheaves
$$U\mapsto \mathrm{Hom}(\Gamma_c(U;\cC_X^\cdot),\CC).$$
Now the Borel-Moore homology of $X$ is defined as 
\beq\label{b11} H_i(X)=H^{-i}\left( \mathrm{Hom}(\Gamma_c(X;\cC_X^\cdot),\CC)\right)
=H^{-i}\left( \mathrm{Hom}(R\Gamma_c(X;\CC),\CC)\right).\eeq

Another way to think of Borel-Moore homology is to consider
the complex of geometric chains. Here a geometric chain refers to a stable limit
of locally finite simplicial chains under refinements of triangulations \cite[I.2.1]{Borel}.
The complex of geometric chains is quasi-isomorphic to $\mathrm{Hom}(\Gamma_c(X;\cC_X^\cdot),\CC)$
and hence its homology coincides with the Borel-Moore homology. 

When $X$ is an algebraic variety, we have the cycle class map
$$h_X:A_*(X)\to H_*(X)$$ which sends an irreducible subvariety $\xi$ of $X$ to the geometric chain of $\xi$ obtained by a triangulation \cite[Chapter 19]{Fulton}. 

\subsection{Functoriality}\label{S2.2}
If $f:X\to Y$ is a proper morphism of varieties, 
the functorial pushforward map
\beq\label{b3} f_*:H_*(X)\lra H_*(Y)\eeq
is defined by sending a geometric chain $c$ on $X$ to its image $f(c)$ on $Y$. 
For an open inclusion $\imath:U\to X$ and its complement $\jmath:Z=X-U\to X$, 
we have the long exact sequence
\beq\label{b4} \cdots\lra H_i(Z)\mapright{\jmath_*}H_i(X)\mapright{\imath^*} H_i(U)\lra H_{i-1}(Z)\lra \cdots\eeq
where $\imath^*$ denotes the restriction of geometric chains to $U$ \cite[IX.2.1]{Iver}.

For a smooth morphism $f:X\to Y$ of algebraic varieties, taking the inverse image of a geometric cycle gives us the pullback homomorphism of intersection homology groups
\beq\label{c23} f^*:\ih^{\mathbf{p}}_i(Y)\lra \ih^{\mathbf{q}}_{i+\dim_\RR X-\dim_\RR Y}(X), \quad \mathbf{p}\le \mathbf{q}\eeq
where $\ih^{\mathbf{p}}_*(Y)$ and $\ih^{\mathbf{q}}_*(X)$ denote the intersection homology with respect to the perversities ${\mathbf{p}}$ and ${\mathbf{q}}$ respectively \cite{GM83}. 

When $\mathbf{p}$ is the top perversity ${\mathbf{t}}(i)=i-2$ and $X$ is normal, 
the intersection homology $\ih^{\mathbf{t}}_i(X)$ is the Borel-Moore homology $H_i(X)$ \cite[V.2.12]{Borel}. 
By \eqref{c23} with $X=Y$, for the middle perversity $\mathbf{m}(2i)=i-1$, we have a canonical homomorphism 
\beq\label{b21}\epsilon_X: \ih_i^{\mathbf{m}}(X)\lra \ih^{\mathbf{t}}_i(X)=H_i(X).\eeq
As \eqref{b21} and \eqref{c23} are defined by taking inverse images by the identity map and $f$ respectively, we have a commutative diagram
\beq\label{c24}\xymatrix{
\ih_i^{\mathbf{m}}(Y)\ar[r]^{f^*} \ar[d]_{\epsilon_Y} &\ih^{\mathbf{m}}_{i+l}(X)\ar[d]^{\epsilon_X}\\
H_i(Y)\ar[r]^{f^*}  &H_{i+l}(X)
}\eeq
where $l=\dim_\RR X-\dim_\RR Y.$

More generally, a morphism $f:X\to Y$ of irreducible algebraic varieties is called \emph{placid} (\cite{GM85}) if there is a stratification of $Y$ such that 
\beq\label{b72}\mathrm{codim}\, T\le \mathrm{codim}\, f^{-1}(T)\eeq
 for each stratum $T$ of $Y$.
For example, flat morphisms are placid.  
\begin{prop}\label{b71} \cite[Proposition 4.1]{GM85} 
Let $\mathbf{p}$ be any perversity and $l=\dim_\RR X-\dim_\RR Y.$
If $f$ is placid, the pullback of generic chains induces a homomorphism in intersection homology
\beq\label{b70} f^*:\ih_i^{\mathbf{p}}(Y)\lra \ih_{i+l}^{\mathbf{p}}(X). \eeq
If $f$ is proper and placid, then we have the pushforward 
\beq\label{c1} f_*:\ih_i^{\mathbf{p}}(X)\lra \ih_{i}^{\mathbf{p}}(Y), \quad [c]\mapsto [f(c)]. \eeq
They fit into the commutative diagrams
\beq\label{c2}
\xymatrix{
\ih^{\mathbf{m}}_i(Y)\ar[r]^{f^*}\ar[d]_{\epsilon_Y} &\ih^{\mathbf{m}}_{i+l}(X)\ar[d]^{\epsilon_X}\\
H_i(Y) \ar[r]^{f^*} & H_{i+l}(X),
}\quad
\xymatrix{
\ih_i^{\mathbf{m}}(X)\ar[d]_{\epsilon_X}\ar[r]^{f_*} & \ih_{i}^{\mathbf{m}}(Y) \ar[d]^{\epsilon_Y}\\
H_i(X)\ar[r]^{f_*} & H_{i}(Y),
}\eeq
\end{prop}
\begin{proof} For reader's convenience, we outline the proof. Since $f$ is analytic, there exists a stratification of $X$ with which $f$ is a stratified map \cite[\S1.2]{GM83}. By McCrory's transversality, a cohomology class $\xi\in \ih_i^{\mathbf{p}}(Y)$ is represented by a geometric chain $c$, dimensionally transversal to any stratum in $X$. Hence the inverse image $f^{-1}(c)$ is $\mathbf{p}$-allowable by \eqref{b72}. The class $f^*(\xi)$ is the class represented by $f^{-1}(c)$.

If $c$ is a geometric chain on $X$ representing a class $[c]\in \ih_i^{\mathbf{p}}(X)$, its image $f(c)$ is $\mathbf{p}$-allowable by the placid condition and hence defines a class $[f(c)]\in \ih_i^{\mathbf{p}}(Y)$.
See \cite{GM85} for details. (See \cite{Kie} for a generalization.) 

The commutativity of the diagrams in \eqref{c2} is obvious from the definitions. 
\end{proof}

From now on, we will mostly use the middle perversity intersection homology and 
denote $\ih^{\mathbf{m}}_i(\cdot)$ by $\ih_i(\cdot)$. 

The pushforward and the pullback are compatible in the following sense: For any fiber diagram 
\[\xymatrix{
X'\ar[r]^u\ar[d]_g &X\ar[d]^f\\
Y'\ar[r]^v & Y
}\]
with $f, g$ placid and $v$ proper, we have
\beq\label{c25} f^*\circ v_*=u_*\circ g^*:H_i(Y')\lra H_{i+l}(X)\eeq
since $f^{-1}(v(c))=u(g^{-1}(c))$ for a geometric chain $c$.

For a projective birational morphism $f:X\to Y$ of normal algebraic varieties, the decomposition theorem of Beilinson, Bernstein, Deligne and Gabber in \cite{BBD} tells us that the middle perversity intersection homology sheaf decomposes as
\beq\label{b25}Rf_*\mathcal{IC}_X\cong \mathcal{IC}_Y\oplus \mathcal{F}\eeq
for a constructible sheaf complex $\mathcal{F}$ on the locus in $Y$ where $f$ is not an isomorphism.
Here $\mathcal{IC}_X$ is the sheaf complex of ${\mathbf{m}}$-allowable chains whose hypercohomology is the middle perversity intersection homology $\ih_*(X)$ of $X$.
By taking the hypercohomology, we have a decomposition
$$ \ih_i(X)\cong \ih_i(Y)\oplus \mathbb{H}_i(\cF)$$
which gives us an injection
\beq\label{b22} \delta:\ih_i(Y)\lra \ih_i(X)\eeq
which restricts to the identity map $\ih_i(U)=\ih_i(f^{-1}U)$ on the open subset $U$ 
of $Y$ over which $f$ is an isomorphism.  
Combining \eqref{b21} and \eqref{b22}, we have a homomorphism
\beq\label{b23} \ih_i(Y)\mapright{\delta} \ih_i(X)\mapright{\epsilon_X} H_i(X).\eeq
The canonical homomorphism \eqref{b21} for $Y$ together with \eqref{b23} give us a (not necessarily commutative) diagram
\beq\label{b28}\xymatrix{
\ih_i(X)\ar[r]^{\epsilon_X} & H_i(X)\ar[d]^{f_*}\\
\ih_i(Y)\ar[r]^{\epsilon_Y}\ar[u]^{\delta} & H_i(Y)
}\eeq
\begin{lemm}\label{b29}
For $\xi\in \ih_i(Y)$, there are $\zeta\in H_i(X)$ and $\eta\in H_i(Z)$ such that 
$$\epsilon_Y(\xi)=f_*(\zeta)+\jmath_*(\eta)\and \epsilon_Y(\xi)|_{U}=\zeta|_{f^{-1}U}$$
where $\jmath:Z\to Y$ is the inclusion of the complement of the open subset $U$ over which $f$ is an isomorphism. 
\end{lemm}
\begin{proof} Let $\zeta=\epsilon_X(\delta(\xi))$.
Then $f_*(\zeta)-\epsilon_Y(\xi)$ is zero when restricted to $U$ because the vertical  arrows in \eqref{b28} are the identity maps when restricted to $U$ and the two horizontal maps coincide with $\epsilon_U$ when restricted to $U$. By \eqref{b4}, $f_*(\zeta)-\epsilon_Y(\xi)$ lies in the image of $\jmath_*$ and the lemma follows.   
\end{proof}

\subsection{Cap product}\label{S2.35}
If $Z\subset X$ is closed, we have the cap product (cf. \cite[IX.3]{Iver})
\beq\label{b5} \cap : H_i(X)\otimes H^p(X,X-Z)\lra H_{i-p}(Z)\eeq
which satisfies nice properties like the projection formula (cf. \cite[IX.3.7]{Iver})
$$f_*(\alpha\cap f^*\eta)=f_*\alpha\cap \eta,\quad \text{for } \alpha\in H_i(X), \eta\in H^p_W(Y)=H^p(Y,Y-W)$$
when $f:X\to Y$ is proper and $W\subset Y$ is closed. 
For closed sets $A$ and $B$ of a topological space $X$, by \cite[IX.3.4]{Iver}, 
\beq\label{c33} (\xi\cap \alpha)\cap \beta|_A=\xi\cap(\alpha\cup \beta), \quad \xi\in H_*(X), \alpha\in H^*_A(X), \beta\in H^*_B(X)\eeq
When $Z=X$, \eqref{b5} is the ordinary cap product.

When $X$ is a closed subset of an oriented (topological) manifold $M$, the cap product by the orientation class $[M]\in H_{\dim_\RR M}(M)$ of $M$ gives us an isomorphism
\beq\label{b6} H^{\dim_\RR M-i}(M,M-X)= H^{\dim_\RR M-i}_X(M)\mapright{\cong} H_i(X).\eeq
In particular, if $X$ is a smooth oriented manifold with its orientation class $[X]$, by taking $M=X$, we have the Poincar\'e duality
\beq\label{b7} [X]\cap :H^{\dim_\RR X-i}(X)\mapright{\cong} H_i(X).\eeq

We have the following lifting property analogous to Lemma \ref{b29}.
\begin{lemm}\label{f19}
Let $\rho_M:\widetilde{M}\to M$ be a proper birational morphism of smooth varieties  
and $\widetilde{X}=\widetilde{M}\times_MX$ for a closed subset $X$ of $M$.
Let $U$ be an open set in $M$ over which $\rho_M$ is an isomorphism. Then for $\xi \in H_i(X)$, there are $\zeta\in H_i(\widetilde{X})$ and $\eta\in H_i(Z)$ such that
\[ \xi=\rho_*\zeta+\jmath_*\eta \and \xi|_{X\cap U}=\zeta|_{\rho^{-1}(X\cap U)} \]
where $\rho=\rho_M|_{\widetilde{X}}:\widetilde{X}\to X$ and $\jmath:Z\to X$ is the inclusion
of $Z=X-U$.
\end{lemm}
\begin{proof}
As $\rho_M$ sends $\widetilde{M}-\widetilde{X}$ to $M-X$, we have a homomorphism
$$H_i(X)\cong H^{\dim_\RR M-i}(M,M-X)\mapright{\rho_M^*} H^{\dim_\RR M-i}(\widetilde{M}, \widetilde{M}-\widetilde{X})\cong H_i(\widetilde{X})$$
by \eqref{b6}. Let $\zeta\in H_i(\widetilde{X})$ be the image of $\xi$ by this map. 
As the map is canonical and $\rho$ is an isomorphism over $X\cap U$, we have $\xi|_{X\cap U}=\zeta|_{\rho^{-1}(X\cap U)}$ 
after identifying $\rho^{-1}(X\cap U)$ with $X\cap U$ by $\rho$. Hence $\xi-\rho_*\zeta|_{X\cap U}$ is zero and thus lies in the
image of $\jmath_*$ by \eqref{b4}. 
\end{proof}

\subsection{Gysin maps}\label{S2.3}
If $f:X\to Y$ is a continuous map of oriented manifolds, then the pullback homomorphism $f^*:H^i(Y)\to H^i(X)$ together with the Poincar\'e duality \eqref{b7} induce the Gysin map 
\beq\label{b8} f^!:H_i(Y)\cong H^{\dim_\RR M-i}(Y)\mapright{\imath^*} H^{\dim_\RR M-i}(X)\cong H_{i+l}(X)\eeq
where $l=\dim_\RR X-\dim_\RR Y$.
More generally, let $g:W\to Y$ be a continuous map and $Z=X\times_Y W$ so that we have a commutative fiber diagram
\[\xymatrix{
Z\ar[r]\ar[d] & X\ar[d]^f\\
W\ar[r]^g & Y.
}\]
Fixing a topological closed embedding $\imath:W\hookrightarrow \RR^n$, we have a fiber diagram 
\[\xymatrix{
Z\ar@{^(->}[r]\ar[d] & X\times \RR^n\ar[d]\\
W\ar@{^(->}[r] & Y\times \RR^n.
}\]
where the bottom arrow is $(g,\imath)$ and the right vertical arrow is $\hat{f}:=(f,1_{\RR^n})$.
Then we have the Gysin map
\beq\label{b16} f^!:H_i(W)\cong H^{\dim_\RR Y+n-i}(Y\times\RR^n,Y\times \RR^n-W)\mapright{\hat{f}^*} \eeq
\[\lra H^{\dim_\RR Y+n-i}(X\times\RR^n,X\times\RR^n-Z)\cong H_{i+l}(Z).\]
For a proper map $u:W'\to W$ and $Z'=Z\times_WW'$ so that we have a fiber diagram
\[\xymatrix{
Z'\ar[r]^v \ar[d] &Z\ar[r]\ar[d] & X\ar[d]^f\\
W'\ar[r]^u &W\ar[r]^g & Y,
}\]
we have 
\beq\label{c30}
f^!\circ u_*=v_*\circ f^!:H_i(W')\lra H_{i+l}(Z).
\eeq
One way to prove \eqref{c30} is to use the definition of the pushforward in \cite[Appendix B]{FulY} and the commutativity of cohomology pullbacks.

\subsection{Euler classes}\label{S2.4}
If $\pi:E\to Y$ is a complex vector bundle of rank $r$ over a topological space $Y$, then
the canonical orientation on the fibers by the complex structure gives us the Thom class
$$\tau_{Y/E}\in H^{2r}(E,E-Y)=H^{2r}_Y(E)$$
where $Y$ is embedded into $E$ by the zero section $0:Y\to E$. The cap product with $\tau_{Y/E}$ gives us the Gysin isomorphism
\beq\label{b12} 0^!_E=(\cdot)\cap\tau_{Y/E} : H_i(E)\lra H_{i-2r}(Y).\eeq
The image of the Thom class $\tau_{Y/E}$ by the pullback $0^*:H^{2r}(E,E-Y)\to H^{2r}(Y)$ is the Euler class $e(E)\in H^{2r}(Y)$. 
More generally, if  $s:Y\to E$ is an arbitrary section whose zero locus is $X=s^{-1}(0)$, then the image of $\tau_{Y/E}$ by the pullback $s^*:H^{2r}(E,E-Y)\to H^{2r}(Y,Y-X)$ is the restricted Euler class 
\beq\label{b13} e(E,s)=s^*(\tau_{Y/E})\in H^{2r}(Y,Y-X)=H^{2r}_X(Y).\eeq
If $0\to E'\to E\to E''\to 0$ is an exact sequence of complex vector bundles and $s$ is a section of $E'$ with $X=s^{-1}(0)$, then we have the Whitney sum formula
\beq\label{b40} e(E,s)=e(E',s)\cup e(E'')|_X.\eeq

The cap product with $e(E,s)$ gives us the \emph{Gysin map} or the \emph{operational Euler class} 
\beq\label{b14} s^!= e^{op}(E,s)=(\cdot)\cap e(E,s):H_i(Y)\lra H_{i-2r}(X).\eeq

Let $E$ be a complex vector bundle over $Y$ of rank $r$ with a section $s$ and let $\rho:\tY\to Y$ be a proper map. 
Let $X=s^{-1}(0)$ and $\widetilde{X}=\tilde{s}^{-1}(0)$ where $\tilde{s}$ is the pullback of $s$ to $\widetilde{E}=\rho^*E$. 
Then from the definition, we have 
$${\rho}^*e(E,s)=\rho^*s^*\tau_{Y/E}=\tilde{s}^*\tau_{\tY/\tE}=e(\widetilde{E},\tilde{s})\in H^{2r}(\tY,\tY-\widetilde{X}).$$
By the projection formula \cite[IX.3.7]{Iver}, we have 
\beq\label{b38} \rho'_*\circ\tilde{s}^!(\xi)=\rho'_*(\xi\cap e(\widetilde{E},\tilde{s}))=\rho_*(\xi)\cap e(E,s)=s^!\circ\rho_*(\xi),\quad \forall \xi\in H_i(\tY)\eeq
where $\rho':\widetilde{X}\to X$ is the restriction of $\rho$. 
\begin{lemm}\label{c3}
Let $u:M'\to M$ be a continuous map of smooth oriented manifolds.  Let $E_M$ be a complex vector bundle on $M$ with a section $s_M$.  For a topological space $Y$ and a continuous map $Y\to M$, consider the fiber diagram
\[\xymatrix{
Y'\ar[r]\ar[d]_v & M'\ar[d]^u\\
Y\ar[r] & M.
}\]
Let  $(E,s)$ (resp. $(E',s')$) denote the pullback of $(E_M,s_M)$ to $Y$ (resp. $Y'$) and $X=\zero(s)$, $X'=\zero(s')=X\times_Y Y'$. 
Then we have 
\beq\label{c4}
(u^!\circ s^!)(\xi)=u^!(\xi\cap e(E,s))=u^!(\xi)\cap e(E', s')=({s'}^!\circ u^!)(\xi)\in H_{i+l}(X')\eeq
for $\xi\in  H_i(Y)$ where $l=\dim_\RR M'-\dim_\RR M$.
\end{lemm}
\begin{proof} Fix a topological closed embedding $\imath:Y\hookrightarrow \RR^n$.
By \eqref{b16},  there is a unique $\alpha\in H^{\dim_\RR M+n-i}(M\times \RR^n,M\times \RR^n-Y)$ satisfying 
$$[M\times\RR^n]\cap \alpha=\xi \and u^!(\xi)=[M'\times \RR^n]\cap {\hat{u}}^*(\alpha)$$
where $\hat{u}=(u,1_{\RR^n}):M'\times\RR^n\to M\times\RR^n$.  Let $(\hat{E},\hat{s})$ denote the pullback of $(E_M,s_M)$ to $M\times \RR^n$. 
Since 
$$\xi\cap e(E,s)=([M\times \RR^n]\cap \alpha)\cap e(E,s)=[M\times \RR^n]\cap (\alpha\cup e(\hat{E},\hat{s})),$$ we have 
$$u^! (\xi\cap e(E,s))=[M'\times \RR^n]\cap {\hat{u}}^*(\alpha\cup e(\hat{E},\hat{s}))=([M'\times \RR^n]\cap {\hat{u}}^*(\alpha))\cap v^*e(E,s).$$
As $e(E,s)$ is defined by pulling back the Thom class, $v^*e(E,s)=e(E', s')$. This proves the lemma.
\end{proof}

Since the cap product is compatible with the open restriction, the pushforward by \eqref{b38} and the boundary operator in \eqref{b4}, for a closed subset $Z$ and its complement $U=Y-Z$, we have the commutative diagram (cf. \cite[IX.3.6]{Iver})\small
\beq\label{b41} \xymatrix{
\cdots\ar[r] & H_i(Z)\ar[r]^{\jmath_*}\ar[d] & H_i(Y)\ar[r]^{\imath^*}\ar[d] & H_i(U)\ar[r]^\partial\ar[d] & H_{i-1}(Z)\ar[d] \\ 
\cdots\ar[r] & H_{i-2r}(Z\cap X)\ar[r]^{\jmath_*} & H_{i-2r}(X)\ar[r]^{\imath^*} & H_{i-2r}(U\cap X)\ar[r]^\partial & H_{i-1-2r}(Z\cap X)
}\eeq \normalsize
where the vertical arrows are the cap products with the Euler classes of $(E,s)$ and its restrictions to $U$ and $Z$. 

Let $E$ and $F$ be complex vector bundles on a topological space $Y$ with sections $s$ and $t$ respectively. 
Then for $\xi\in H_i(Y)$, we have
\beq\label{b80}
\xi\cap e(E,s)\cap e(F|_{s^{-1}(0)},t|_{s^{-1}(0)})=\xi\cap e(E\oplus F,s\oplus t)\eeq
\[=\xi\cap e(F,t)\cap e(E|_{t^{-1}(0)},s|_{t^{-1}(0)})\]
by \cite[IX.3.4]{Iver}.

When $Y$ is an oriented manifold and the image of $s$ in $E$ is transversal to the zero section $0(Y)$ so that $X=s^{-1}(0)$ is an oriented submanifold of real codimension $2r$ with the inclusion map $u:X\to Y$, we have the equality \beq\label{c28} u^!=s^!=e^{op}(E,s):H_i(Y)\lra H_{i-2r}(X)\eeq 
where $u^!$ is the Gysin map defined in \eqref{b8} because the normal bundle to $X$ in $Y$ is isomorphic to $E|_X$ and $[X]=[Y]\cap e(E,s)$ by \cite[IX.4.8]{Iver}. 

By the same argument, if $X=s^{-1}(0)$ is the zero locus of a section $s$ of an algebraic vector bundle $E$ of rank $r$ over a (not necessarily smooth) irreducible algebraic variety $Y$ and $c$ is a geometric chain in $Y$ representing a class $[c]\in H_i(Y)$, which is dimensionally transversal to each stratum of $X$ for a stratification of $X$ which makes the inclusion $u:X\to Y$ a stratified map, then $u^{-1}(c)=c\cap X$ represents the class $s^![c]=[c]\cap e(E,s)$ in $H_{i-2r}(X)$ when $2r=\dim_\RR Y-\dim_\RR X$. By the proof of Proposition \ref{b71} above, if the inclusion $u$ is furthermore placid, we have 
\beq\label{c29} u^*=s^!=e^{op}(E,s):H_i(Y)\lra H_{i-2r}(X).\eeq 
Combining \eqref{c29} with the first diagram in \eqref{c2}, we have 
a commutative diagram
\beq\label{b73}\xymatrix{
\ih_i(Y)\ar[rr]^{u^*} \ar[d]_{\epsilon_Y} && \ih_{i-2r}(X)\ar[d]^{\epsilon_X}\\
H_i(Y)\ar[rr]^{s^!=e^{op}(E,s)} && H_{i-2r}(X).
}\eeq


\begin{rema}\label{b15}
When $Y$ is a smooth algebraic variety and $s\in H^0(Y;E)$ is a regular section of an algebraic vector bundle $E$ over $Y$ of rank $r$, letting $X=s^{-1}(0)$ denote the closed subscheme defined by the image $I$ of $s^\vee:E^\vee\to \sO_Y$, we have a perfect obstruction theory 
$$[E^\vee|_X\mapright{d\circ s^\vee} \Omega_Y|_X]$$ of $X$ where $d:I/I^2\to \Omega_Y|_X$ is the differential. Then we have a closed embedding 
$$C_{X/Y}=\mathrm{Spec}_X(\bigoplus_n I^n/I^{n+1})\subset E|_X$$ by the surjection 
$\mathrm{Sym}(s^\vee):\mathrm{Sym}(E^\vee|_X)\to \bigoplus_n I^n/I^{n+1}$ and the virtual fundamental class $$[X]\virt=0^!_{E|_X}[C_{X/Y}]\in A_*(X).$$ By \cite[Chapter 19]{Fulton}, the image of the virtual fundamental class $[X]\virt$ by the cycle class map $h_X:A_*(X)\to H_*(X)$ coincides with $s^![Y]=[Y]\cap e(E,s)$ where $[Y]\in H_{\dim_\RR Y}(Y)$ is the orientation class (fundamental class) of $Y$.  
\end{rema}

\subsection{Borel-Moore homology of a quasi-homogeneous cone}\label{S2.5}

Let $B=\CC^m$ ($m\ge 2$) be an affine space equipped with a quasi-homogeneous polynomial $w$ whose zero locus $Z=w^{-1}(0)$ has isolated singularity at the origin. Let $\PP Z=Z-\{0\}/\CC^*$. Then we have
\beq\label{p1} H_m(Z)\cong H^{m-2}_{\mathrm{prim}}(\PP Z).\eeq
In fact, if $\bar{\xi}$ is a geometric chain representing a class $[\bar{\xi}]$ in $H^{m-2}_{\mathrm{prim}}(\PP Z)$, the closure in $Z$ of the inverse image of $\bar{\xi}$ by the quotient map $Z-\{0\}\to Z-\{0\}/\CC^*=\PP Z$ defines a cycle $\xi$ representing the class $[\xi]\in H_m(Z)$ via the isomorphism \eqref{p1}. By \cite[\S2.3]{Voisin}, there are real analytic subsets $\xi_1,\cdots,\xi_{b_m}$ of $Z$ which represent a basis $\{[\xi_1],\cdots, [\xi_{b_m}]\}$ of $H_m(Z)$ where $b_m=\dim H_m(Z).$

The middle perversity intersection homology of $Z$ is (cf. \cite[p.20]{Borel})
\beq\label{d11}\ih_i(Z)=\left\{ \begin{matrix} H^{m-2}_{\mathrm{prim}}(\PP Z) & i=m\\
\CC & i=2m-2 \\
0 & i\ne m, 2m-2 \end{matrix}  \right.\eeq
so that \eqref{b21} becomes an isomorphism
\beq\label{b24} \ih_m(Z)\cong H_m(Z). \eeq
The middle homology groups \eqref{b24} 
will be the (broad) state space of our cohomological field theory in \S\ref{S4}. 

When $w_1$ and $w_2$ are quasi-homogeneous polynomials on $B_1=\CC^{m_1}$ and $B_2=\CC^{m_2}$ respectively, we have a quasi-homogeneous polynomial 
$$w=w_1\boxplus w_2:B=B_1\times B_2\lra \CC, \quad w(x,y)=w_1(x)+w_2(y)$$ and the inclusion map $\imath:w_1^{-1}(0)\times w_2^{-1}(0)\to w^{-1}(0)$. 
The cross product (cf. \cite[IX.5.7]{Iver}) together with the pushforward $\imath_*$ give us a homomorphism 
\beq\label{b17} H_i(w_1^{-1}(0))\otimes H_j(w_2^{-1}(0))\lra H_{i+j}(w_1^{-1}(0)\times w_2^{-1}(0))\mapright{\imath_*} H_{i+j}(w^{-1}(0))\eeq
which we will sometimes call the \emph{Thom-Sebastiani map}.

\subsection{Borel-Moore homology of \DM stacks}\label{S2.6}

All of the above definitions and properties from \S\ref{S2.1} to \S\ref{S2.4} extend to  \DM stacks, whose coarse moduli spaces are locally compact Hausdorff countable CW complexes.
Indeed, all results are obtained by sheaf theory and local analysis. The pullback of a geometric chain by an \'etale map is a geometric chain. Hence the sheaf complex of geometric chains is well defined on a \DM stack and all the definitions and results hold for \DM stacks. For instance, the decomposition theorem \eqref{b25} proved for schemes in \cite{BBD} comes from the degeneration of a spectral sequence of sheaf complexes using perverse filtration. Since \'etale maps preserve the perverse filtration, the sheaf theoretic statement \eqref{b25} holds for \DM stacks as well.

\bigskip

\def\fZ{{\mathfrak{Z}}}
\def\tfZ{{\widetilde{\fZ}}}

\section{Cosection localized Gysin maps}\label{S3}

In this section, we generalize the Gysin map \eqref{b14} to the case when $E$ admits a cosection. 
Our construction parallels those of the cosection localized Gysin maps in \cite[\S2]{KLc} and \cite{KLk}.

\subsection{A localized Gysin map}\label{S3.1}

The following is the common set-up in this section.
\begin{assu}\label{c5}
Let $Y$ be a normal \DM stack over $\CC$ and $s$ be a section of an algebraic vector bundle $E$ of rank $r$ over $Y$ whose zero locus is denoted by $X=s^{-1}(0)$. Let $\sigma:E\to \sO_Y$ be a homomorphism of coherent sheaves on $Y$ such that \beq\label{b77} \sigma\circ s=0.\eeq
Let $S=X\times_Y \fZ$
where $\fZ=\sigma^{-1}(0)=\zero(\sigma)$ is the degeneracy locus where $\sigma$ is zero (i.e. not surjective). 
\end{assu}

%
%
\begin{theo}\label{b18}
Under Assumption \ref{c5}, we have a homomorphism 
\beq\label{b19} s^!_\sigma= e^{op}_{\sigma}(E,s):\ih_i(Y) \lra H_{i-2r}(S)\eeq
whose composition with $\imath_*:H_{i-2r}(S)\to H_{i-2r}(X)$ equals 
$s^!\circ \epsilon_Y $ where $\imath:S\to X$ denotes the inclusion map.
\end{theo}

\begin{proof}
Let $I\subset \sO_Y$ denote the image of $\sigma$ that defines ${\mathfrak{Z}}=\sigma^{-1}(0)$. Let $$\rho:\tY\lra Y$$ be the normalization of the blowup of $Y$ along the ideal $I$ so that the pullback of $\sigma$ to $\tY$ is surjective
$$\tsig: \tE=\rho^*E\twoheadrightarrow \sO_{\tY}(-\tfZ)$$
for an effective Cartier divisor $\tfZ$ of $\tY$ lying over $\fZ=\zero(\sigma)$. Such a morphism $\rho$ is called \emph{$\sigma$-regularizing} and we will show below that our construction is independent of the choice of a $\sigma$-regularizing morphism.

By Lemma \ref{b29}, for $\xi\in \ih_i (Y)$, we have $\zeta\in H_i(\tY)$ and $\eta\in H_i(\fZ)$ such that 
\beq\label{b30} \epsilon_Y(\xi)=\rho_*(\zeta) + \jmath_*(\eta) \in H_i(Y)\eeq
where $\jmath:\fZ\to Y$ is the inclusion map. 
 
For $\eta$, we have 
\beq\label{b31} \eta\cap e(E|_\fZ,s|_\fZ)\in H_{i-2r}(S)\eeq 
by the cap product with the Euler class $e(E|_\fZ,s|_\fZ)\in H^{2r}(\fZ,\fZ-S)$. 

For $\zeta$, let $F=\mathrm{ker}(\tsig)$ so that we have a short exact sequence
\beq\label{b47} 0\lra F\lra \tE\mapright{\tsig} \sO_{\tY}(-\tfZ)\lra 0.\eeq 
Since $\sigma\circ s=0$, the pullback $\tilde{s}\in H^0(\tY;\widetilde{E})$ of the section $s\in H^0(Y;E)$ is in fact a section of $F$. Hence we have
\beq\label{b32} \zeta\cap e(F,\tilde{s})\in H_{i-2r+2}(\tX)\eeq
where $\widetilde{X}=X\times_Y\tY=\zero(\tilde{s})$. 
Moreover, if we restrict the canonical section $t$ of $\sO_{\tY}(\tfZ)$ defining the effective Cartier divisor $\tfZ$ to $\widetilde{X}$, we have 
\beq\label{b33}
\zeta\cap e(F,\tilde{s})\cap e(\sO_{\widetilde{X}}(\tS), t_{\tS})\in H_{i-2r}(\tS) \eeq 
where $\tS=\tfZ\times_{\tY}\widetilde{X}$ and $t_{\tS}$ is the restriction of $t$ to $\tX$. 
Let $\rho_{S}:\tS\to S$ denote the restriction of $\rho$ to $\tS$. Then $\rho_{\tS}$ is proper as $\rho$ is proper. Applying $-{\rho_{S}}_*$ to \eqref{b33}, we obtain
\beq\label{b34} -{\rho_{S}}_*\left(\zeta\cap e(F,\tilde{s})\cap e(\sO_{\widetilde{X}}(\tS), t_{\tS})\right)\in H_{i-2r}(S).\eeq

By adding \eqref{b31} and \eqref{b34}, we define 
\beq\label{b35} s^!_\sigma(\xi)=e^{op}_{\sigma}(E,s) (\xi)=
-{\rho_{S}}_*\left(\zeta\cap e(F,\tilde{s})\cap e(\sO_{\widetilde{X}}(\tS), t_{\tS})\right) + \eta\cap e(E|_\fZ,s|_\fZ).\eeq

\def\hdel{\hat{\delta}}
\def\hze{\hat{\zeta}}
\def\het{\hat{\eta}}
\def\tjm{\tilde{\jmath}}

To show that \eqref{b35} is well defined, we check that $s^!_\sigma(\xi)=e^{op}_{\sigma}(E,s) (\xi)$ is independent of the choice of $\zeta$ and $\eta$ as long as they satisfy \eqref{b30} and 
\beq\label{b76}\zeta|_{\tY-\tfZ}=\epsilon_Y(\xi)|_{Y-\fZ}\in H_i(Y-\fZ)\eeq via the identification $\tY-\tfZ=Y-\fZ$.\footnote{\eqref{b76} is guaranteed by Lemma \ref{b29}.} 
First fix $\zeta$ and pick another $\het\in H_i(\fZ)$ such that $\jmath_*\eta=\jmath_*\het$. Then \eqref{b41} gives us the commutative diagram 
\[\xymatrix{
H_{i+1}(Y-\fZ)\ar[r]^\partial\ar[d]_{e(E|_{Y-\fZ},s|_{Y-\fZ})} & H_i(\fZ)\ar[r]^{\jmath_*}\ar[d]^{e(E|_\fZ,s|_\fZ)} & H_i(Y)\ar[d]^{e(E,s)}\\
H_{i+1-2r}(X-S)\ar[r]^\partial & H_{i-2r}(S)\ar[r]^{\jmath_*} & H_{i-2r}(X).
}\]
Since $\sigma\circ s=0$, $s|_{Y-\fZ}$ is a section of the kernel $\hat{E}$ of the surjection $\sigma:E|_{Y-\fZ}\to \sO_{Y-\fZ}$ and hence the first vertical arrow is $$e(E|_{Y-\fZ},s|_{Y-\fZ})=e(\hat{E},s|_{Y-\fZ})\cup e(\sO_{X-S})=0$$
by \eqref{b40} 
since the Euler class of the trivial bundle is zero. Since $\jmath_*(\eta-\het)=0$, $\eta-\het=\partial (\alpha)$ for some $\alpha\in H_{i+1}(Y-\fZ)$. Therefore, we have
\beq\label{b43} (\eta-\het)\cap e(E|_\fZ,s|_\fZ)=\partial (\alpha\cap e(E|_{Y-\fZ},s|_{Y-\fZ}))= 0.\eeq
Thus $e^{op}_{\sigma}(E,s) (\xi)$ is independent of the choice of $\eta$. 

Suppose we picked another $\hat{\zeta}$ satisfying 
$$\hat{\zeta}|_{\tY-\tfZ}=\epsilon_Y(\xi)|_{Y-\fZ}={\zeta}|_{\tY-\tfZ}.$$
By the above paragraph, we may choose any $\het$ satisfying 
\beq\label{b44} \jmath_*(\het)+\rho_*(\hze)=\epsilon_X(\xi)=\jmath_*(\eta)+\rho_*(\zeta).\eeq
Since $\zeta|_{\tY-\tfZ}=\hze|_{\tY-\tfZ}$, by \eqref{b4}, there exists a $\lambda\in H_i(\tfZ)$ such that 
$$\tilde{\jmath}_*\lambda=\zeta-\hze$$ 
where $\tilde{\jmath}:\tfZ\to \tY$ denotes the inclusion map. 
So we let 
\beq\label{b45} \het=\eta+{\rho_\fZ}_*(\lambda)\eeq
where $\rho_\fZ:\tfZ\to \fZ$ is the restriction of $\rho$ to $\fZ$. 
Then it is straightforward to see that \eqref{b44} is satisfied. 
By applying \eqref{b38} repeatedly and $t_{\tS}|_{\tS}=0$ as well as the exact sequence
$$0\lra F|_{\tX}\lra \tE|_{\tX}\lra \sO_{\tX}(-\tS)\lra 0$$
from \eqref{b47}, 
we have
\beq\label{b46}
{\rho_{S}}_*\left((\zeta-\hze)\cap e(F,\tilde{s})\cap e(\sO_{\widetilde{X}}(\tS), t_{\tS})\right)\eeq 
\[={\rho_{S}}_*\left(\tjm_*\lambda\cap e(F,\tilde{s})\cap e(\sO_{\widetilde{X}}(\tS), t_{\tS})\right)\] 
\[={\rho_{S}}_*\left(\tilde{\imath}_*(\lambda\cap \tjm^* e(F,\tilde{s}))\cap e(\sO_{\widetilde{X}}(\tS), t_{\tS})\right)\] 
\[={\rho_{S}}_*\left(\lambda\cap \tjm^* e(F,\tilde{s})\cap \tilde{\imath}^* e(\sO_{\widetilde{X}}(\tS))\right)\]
\[=-{\rho_{S}}_*\left(\lambda\cap \tjm^* e(F,\tilde{s})\cap \tilde{\imath}^* e(\sO_{\widetilde{X}}(-\tS))\right)\]
\[=-{\rho_{S}}_*\left(\lambda\cap \tjm^* e(\tE,\tilde{s})\right)
=-{\rho_{S}}_*\left(\lambda\cap \tjm^* \rho^*e(E,{s})\right)\]
\[=-{\rho_{S}}_*\left(\lambda\cap {\rho_\fZ}^* \jmath^*e(E,{s})\right)
=-{\rho_{\fZ}}_*\lambda\cap \jmath^*e(E,{s})\]
\[=(\eta-\het)\cap \jmath^*e(E,{s})=(\eta-\het)\cap e(E|_\fZ,{s}|_\fZ)\]
where $\tilde{\imath}:\tS\to \tX$ denotes the inclusion. By \eqref{b46}, $s^!_\sigma(\xi)$ defined by \eqref{b35} is independent of the choices of $\zeta$ and $\eta$. Moreover, by Lemma \ref{b50} below, $s^!_\sigma(\xi)$ defined by \eqref{b35} is independent of the choice of the resolution $\rho:\tY\to Y$ of the degeneracy of $\sigma$. Hence $s^!_\sigma=e^{op}_\sigma(E,s)$ is well defined, independent of all the choices.

Finally we prove that 
\beq\label{b36} \imath_*\circ s^!_\sigma=s^!\circ \epsilon_Y.\eeq
For $\xi, \zeta, \eta$ satisfying \eqref{b30}, applying the projection formula \eqref{b38} repeatedly as well as \eqref{b47} and \eqref{b35}, we have
$$s^!\circ \epsilon_Y (\xi)=\rho_*(\zeta)\cap e(E,s)+\jmath_*(\eta)\cap e(E,s)$$
$$={\rho_X}_*(\zeta\cap e(\tE,\tilde{s}))+\imath_*(\eta\cap e(E|_\fZ,s|_\fZ))$$
$$={\rho_X}_*(\zeta\cap e(F,\tilde{s})\cap e(\sO_{\tX}(-\tS)))+\imath_*(\eta\cap e(E|_\fZ,s|_\fZ))$$
$$=-{\rho_X}_*(\zeta\cap e(F,\tilde{s})\cap e(\sO_{\tX}(\tS)))+\imath_*(\eta\cap e(E|_\fZ,s|_\fZ))$$
$$=-{\rho_X}_*\tilde{\imath}_*(\zeta\cap e(F,\tilde{s})\cap e(\sO_{\tX}(\tS),t_{\tS}))+\imath_*(\eta\cap e(E|_\fZ,s|_\fZ))$$
$$=-{\imath}_*{\rho_S}_*(\zeta\cap e(F,\tilde{s})\cap e(\sO_{\tX}(\tS),t_{\tS}))+\imath_*(\eta\cap e(E|_\fZ,s|_\fZ))$$
$$={\imath}_*\left(-{\rho_S}_*(\zeta\cap e(F,\tilde{s})\cap e(\sO_{\tX}(\tS),t_{\tS}))+\eta\cap e(E|_\fZ,s|_\fZ)\right)$$
$$={\imath}_*\circ s^!_\sigma(\xi)$$
where $\rho_X:\tX\to X$ is the restriction of $\rho$ to $\tX$. This proves \eqref{b36}.
\end{proof}

\begin{rema}\label{e50}
In Assumption \ref{c5}, we required $Y$ to be normal.
If $Y$ is not normal, let $Y^{nor}\to Y$ denote the normalization of $Y$.
Since $\ih_*(Y^{nor})\cong \ih_*(Y)$ by \cite[4.2]{GM80}, we can replace $Y$ by $Y^{nor}$ and
pull back $E, s, \sigma$ etc to $Y^{nor}$.
Theorem \ref{b18} then gives us the cosection localized Gysin map
$$s^!_\sigma: \ih_i(Y)\cong \ih_i(Y^{nor})\lra H_{i-2r}(S^{nor})\lra H_{i-2r}(S)$$
where $S^{nor}=S\times_YY^{nor}$ and the last arrow is the proper pushforward.
Therefore the normality in Assumption \ref{c5} is only a matter of notational convenience. 
\end{rema}

\subsection{Remarks}\label{c6}
A key ingredient in the construction of the cosection localized Gysin map $s^!_\sigma=e^{op}_\sigma(E,s)$ above is the decomposition \eqref{b30}. For a class in $H_*(Y)$, there is no guarantee that we have a decomposition like \eqref{b30} with \eqref{b76}. For classes coming from the middle perversity intersection homology, the decomposition theorem in \cite{BBD} enables us to find such a decomposition (Lemma \ref{b29}) and that is why we use $\ih_*(Y)$ instead of $H_*(Y)$ as the starting point of the map $s^!_\sigma=e^{op}_\sigma(E,s)$.

In fact, the proofs of Theorem \ref{b18} and Lemma \ref{b50} tell us that whenever we can write a class $\bar\xi \in H_i(Y)$ as 
$$\bar\xi=\rho_*(\zeta)+\jmath_*(\eta)\in H_i(Y), \quad \bar\xi|_{Y-\fZ}=\zeta|_{\tY-\tfZ}$$
 for some resolution $\rho:\tY\to Y$ of the degeneracy of $\sigma$,
the cosection localized Gysin map \eqref{b35} is well defined, independent of choices. We call such a class $\bar\xi\in H_i(Y)$ \emph{$\sigma$-liftable}.

For example, classes in $H_i(Y)$ coming from intersection homology are $\sigma$-liftable by Lemma \ref{b29}. Moreover, classes coming from the Chow group $A_*(Y)$ are $\sigma$-liftable: If $\tilde{\xi}\in A_*(Y)$, $$\tilde{\xi}|_{Y-\fZ}\in A_*(Y-\fZ)=A_*(\tY-\tfZ)$$ extends to a class $\tilde{\zeta}\in A_*(\tY)$ since the restriction map $A_*(\tY)\to A_*(\tY-\tfZ)$ is surjective (cf. \cite[Proposition 1.8]{Fulton}). Then $\tilde{\xi}-\rho_*(\tilde{\zeta})\in \jmath_*A_*(\fZ)$ and hence we have 
\beq\label{c16} \tilde{\xi}=\rho_*\tilde{\zeta}+\jmath_*\tilde{\eta} \quad \text{for some }\tilde{\eta}\in A_*(\fZ).\eeq  
By applying the cycle class map $h_Y:A_*(Y)\to H_*(Y)$ to \eqref{c16}, we obtain a decomposition 
\beq\label{c15} h_Y(\tilde{\xi})=\rho_*(h_{\tY}(\tilde{\zeta}))+\jmath_*(h_\fZ(\tilde{\eta}))\eeq
like \eqref{b30} because the cycle class map is compatible with the pushforward (cf. \cite[Chapter 19]{Fulton}). 

Using \eqref{c16}, we can define the cosection localized Gysin map  similarly as in \eqref{b35} as the composition  
$$A_*(Y)\lra A_*(E|_X(\sigma))\mapright{0^!_{E,\sigma}} A_{*-r}(S)$$ 
where the last arrow is defined in \cite[Corollary 2.9]{KLc}, while the first arrow is the map sending an effective cycle $V$ in $A_*(Y)$ to the normal cone $C_{V\cap X/V}$. As $\sigma\circ s=0$, this cone has support contained in 
$$E|_X(\sigma)=E|_S\cup \mathrm{ker}(E|_{X-S}\to \sO_{X-S}).$$
 
By comparing the proof of \cite[Corollary 2.9]{KLc} with the proof of Theorem \ref{b18}, we find that 
when $\epsilon_Y(\xi)=h_Y(\tilde{\xi})$ for $\xi\in \ih_*(Y)$ and $\tilde{\xi}\in A_*(Y)$, 
\beq\label{c7} s^!_\sigma(\xi)=h_S(0^!_{E,\sigma}([C_{\tilde{\xi}\cap X/\tilde{\xi}}])).\eeq

When $Y$ is smooth, the differential of $s$ gives us the perfect obstruction theory 
$$[T_Y|_X\mapright{ds} E|_X]^\vee$$
equipped with a cosection $\sigma:E\to \mathrm{coker}(ds)\to \sO_X.$
By \cite{KLc}, we have the cosection localized virtual class 
\beq\label{c8} s^!_\sigma[Y]=h_S(0^!_{E,\sigma}[C_{X/Y}])=h_S([X]\virtloc)\eeq
where $[X]\virtloc=0^!_{E,\sigma}[C_{X/Y}]$ denotes the cosection localized virtual cycle defined in \cite[Theorem 5.1]{KLc}

In \S\ref{S3.v}, we will see that if we strengthen Assumption \ref{c5} a bit (cf. Assumption \ref{f21}), all homology classes on $Y$ lift 
to some $\sigma$-regularization $\widetilde{Y}\to Y$ by using Lemma \ref{f19} and thus the cosection localized Gysin map 
$$s^!_\sigma:H_i(Y)\lra H_{i-2r}(S)$$
is defined for all homology classes on $Y$.

\subsection{The Gysin map is canonical} \label{S3.3}

In the proof of Theorem \ref{b18}, we used the blowup of $Y$ along $\fZ$ to resolve the degeneracy of $\sigma.$
But a close examination of the proof reveals that \eqref{b35} is well defined for any birational proper morphism $\bar{\rho}:\bar{Y}\to Y$ such that the pullback $\bar{\sigma}:\bar{\rho}^*E\to \sO_{\bar{Y}}(-\bar{\fZ})$  of $\sigma$ is surjective for an effective Cartier divisor $\bar{\fZ}$ of $\bar{Y}$. Indeed, any such resolution of $\sigma$ gives us the same $s^!_\sigma$. Namely, we may use any resolution of $\sigma$ instead of the blowup of $Y$ along $\fZ$. 
\begin{lemm}\label{b50}
The cosection localized Gysin map $s^!_\sigma=e^{op}_\sigma(E,s)$ in Theorem \ref{b18} is independent of the choice of a resolution $\rho:\tY\to Y$ of the degeneracy of $\sigma$.  
\end{lemm}
\begin{proof} 
By taking the fiber product of two resolutions of $\sigma$, we only need to consider the case where we have two proper morphisms 
$$\bar{Y}\mapright{\mu} \tY\mapright{\rho} Y$$
which are isomorphisms over the surjective locus $Y-\fZ$ of $\sigma$ and resolve the degeneracy of $\sigma$. Our goal is to show that \eqref{b35} is the same for $\tY$ and $\bar{Y}$. 

We use  the notations in the proof of Theorem \ref{b18}.
By the decomposition theorem \eqref{b25}, we fix injective homomorphisms
\beq\label{b52}
\ih_i(Y)\mapright{\delta} \ih_i(\tY)\mapright{\gamma} \ih_i(\bar{Y}).
\eeq
Let $\xi\in \ih_i(Y)$. By the proof of Lemma \ref{b29}, there exists $\eta\in H_i(\fZ)$ such that
\beq\label{b53} \zeta=\epsilon_{\tY}\delta(\xi),\quad \epsilon_Y(\xi)=\rho_*\zeta+\jmath_*\eta.\eeq
By applying Lemma \ref{b29} to $\delta(\xi)\in \ih_i(\tY)$, there is a $\lambda\in H_i(\tfZ)$ such that
\beq\label{b54}
\zeta=\epsilon_{\tY}\delta(\xi)=\mu_*\epsilon_{\bar{Y}}\gamma\delta(\xi)+\tilde{\jmath}_*(\lambda).\eeq 
By \eqref{b53} and \eqref{b54}, we have 
\[ \epsilon_Y(\xi)=(\rho\mu)_*\epsilon_{\bar{Y}}\gamma\delta(\xi)+\rho_*\tilde{\jmath}_*(\lambda)+\jmath_*\eta\]
\[=(\rho\mu)_*\epsilon_{\bar{Y}}\gamma\delta(\xi)+\jmath_*({\rho_\fZ}_*(\lambda)+\eta).\]
Hence, we may let
\beq\label{b55} \bar{\zeta}=\epsilon_{\bar{Y}}\gamma\delta(\xi)\in H_i(\bar{Y}),\quad 
\bar\eta={\rho_\fZ}_*(\lambda)+\eta\in H_i(\fZ)\eeq
so that $\epsilon_Y(\xi)=(\rho\mu)_*\bar{\zeta}+\jmath_*\bar\eta$ and 
\beq\label{b58} \zeta=\mu_*\bar{\zeta}+\tilde{\jmath}_*(\lambda).\eeq 

By its definition, $e^{op}_\sigma(E,s)$ using $\rho:\tY\to Y$ is 
\beq\label{b56}
-{\rho_{S}}_*\left(\zeta\cap e(F,\tilde{s})\cap e(\sO_{\widetilde{X}}(\tS), t_{\tS})\right) + \eta\cap e(E|_\fZ,s|_\fZ)\eeq
while $e^{op}_\sigma(E,s)$ using  $\rho\mu:\bar{Y}\to Y$ is 
\beq\label{b57}
-(\rho_{S}\mu_S)_*\left(\bar{\zeta}\cap e(\bar{F},\bar{s})\cap e(\sO_{\bar{X}}(\bar{S}), t_{\bar{S}})\right) + \bar{\eta}\cap e(E|_\fZ,s|_\fZ)\eeq
where $\bar{E}, \bar{F}, \bar{s}$ are the pullbacks of $\tE, F, \tilde{s}$ respectively to $\bar{Y}$ by $\mu$ and $\mu_S$ denotes the restriction of $\mu$ to $\mu^{-1}(\tilde{S})$.  

By \eqref{b55} and the projection formula \eqref{b38}, the difference of \eqref{b56} and \eqref{b57} is 
$${\rho_{S}}_*\left((\mu_*\bar{\zeta}-\zeta)\cap e(F,\tilde{s})\cap e(\sO_{\widetilde{X}}(\tS), t_{\tS})\right) - {\rho_\fZ}_*(\lambda)\cap e(E|_\fZ,s|_\fZ)$$
$$=-{\rho_{S}}_*\left(\tilde{\jmath}_*(\lambda)\cap e(F,\tilde{s})\cap e(\sO_{\widetilde{X}}(\tS), t_{\tS})\right) - {\rho_\fZ}_*(\lambda)\cap e(E|_\fZ,s|_\fZ)$$
$$=-{\rho_{S}}_*\left(\lambda\cap e(F|_{\tfZ},\tilde{s}|_{\tfZ})\cap e(\sO_{\widetilde{S}}(\tS))\right) - {\rho_\fZ}_*(\lambda)\cap e(E|_\fZ,s|_\fZ)$$
$$={\rho_{S}}_*\left(\lambda\cap e(F|_{\tfZ},\tilde{s}|_{\tfZ})\cap e(\sO_{\widetilde{S}}(-\tS))\right) - {\rho_\fZ}_*(\lambda)\cap e(E|_\fZ,s|_\fZ)$$
$$={\rho_{S}}_*\left(\lambda\cap e(\tE|_{\tfZ},\tilde{s}|_{\tfZ})\right) - {\rho_\fZ}_*(\lambda)\cap e(E|_\fZ,s|_\fZ)$$
$$={\rho_{\fZ}}_*(\lambda)\cap e(E|_{\fZ},{s}|_{\fZ}) - {\rho_\fZ}_*(\lambda)\cap e(E|_\fZ,s|_\fZ)=0.$$
This proves the lemma.
\end{proof}

\subsection{First properties}\label{S3.4}

In this subsection, we prove a few useful properties of the cosection localized Gysin map that will play key roles for cohomological field theories in \S\ref{S4}. 

\begin{prop}\label{c9}
Under Assumption \ref{c5}, 
we further suppose that there is a morphism $u:Y'\to Y$ which is both proper and placid. 
Let $E', s', \sigma', X', \fZ', S'$, etc denote the pullbacks (by fiber products) of $E, s, \sigma, X, \fZ, S$, etc to $Y'$.  
Let $u_S:S'\to S$ and $u_\fZ$ etc denote the pullbacks of $u$.  
Then we have
\beq\label{c10}
{u_S}_*\circ {s'}_{\sigma'}^!=s_\sigma^!\circ u_*:\ih_i(Y')\lra H_{i-2r}(S).
\eeq
\end{prop}

\begin{proof}
By \eqref{c2}, we have a commutative diagram
\beq\label{c11}\xymatrix{
\ih_i(Y')\ar[r]^{\epsilon_{Y'}} \ar[d]_{u_*} & H_i(Y')\ar[d]^{u_*}\\
\ih_i(Y)\ar[r]^{\epsilon_Y} & H_i(Y).
}\eeq
Pick a resolution $\rho:\tY\to Y$ of the degeneracy of $\sigma$ so that we have an exact sequence \eqref{b47} for an effective Cartier divisor $\tfZ$. Consider the fiber product
\[\xymatrix{
\tY'\ar[r]^{u_{\tY}} \ar[d]_{\rho'} & \tY\ar[d]^\rho \\
Y'\ar[r]^u & Y.
}\]
Pick $\zeta'\in H_i(\tY')$ and $\eta'\in H_i(\fZ')$ such that
\beq\label{c12}\epsilon_{Y'}(\xi')=\rho'_*(\zeta')+\jmath'_*(\eta')\eeq
where $\jmath':\fZ'\to Y'$ denotes the inclusion map. By \eqref{c11}, applying $u_*$ to  \eqref{c12}, we have
\[ \epsilon_Y(u_*\xi')=\rho_*({u_{\tY}}_*\zeta')+\jmath_*({u_\fZ}_*\eta')\]
and hence we may use $\zeta={u_{\tY}}_*\zeta'$ and $\eta={u_\fZ}_*\eta'$ for the computation of $s^!_\sigma(u_*\xi')=e^{op}_\sigma(E,s)(u_*\xi')$. 

By applying \eqref{b38} repeatedly, we have
\beq\label{c13}
{u_\fZ}_*(\eta')\cap e(E|_\fZ,s|_\fZ)={u_S}_*\left( \eta'\cap e(E|_{\fZ'},s|_{\fZ'}) \right),
\eeq
\beq\label{c14}
{u_{\tY}}_*(\zeta')\cap e(F,\tilde{s})\cap e(\sO_{\tX}(\tS),t_{\tS})
={u_S}_*\left( \zeta'\cap e(F',\tilde{s}')\cap e(\sO_{\tX'}(\tS'),t_{\tS'}) \right).
\eeq
Combining \eqref{c13} and \eqref{c14} with the definition of $s^!_\sigma=e^{op}_\sigma(E,s)$ and ${s'}^!_{\sigma'}=e^{op}_{\sigma'}(E',s')$ from \eqref{b35}, we obtain \eqref{c10}.
\end{proof}

\begin{prop}\label{b60}
Under Assumption \ref{c5}, 
we further suppose that there is a locally free sheaf $L$ of rank $l$ on $Y$ and a section $\ell$ with $Y'=\zero(\ell)$.
Suppose $\ell$ is transversal to the zero section of $L$ and the inclusion map $u:Y'\to Y$ is placid so that the pullback $u^*$ equals $\ell^!=e^{op}(L,\ell)$. 
Let $E', s', \sigma', X', \fZ', S'$, etc denote the pullbacks (by fiber products) of $E, s, \sigma, X, \fZ, S$, etc to $Y'$.  
Let $u_S:S'\to S$ and $u_\fZ$ etc denote the pullbacks of $u$.  
Then we have a commutative diagram
\beq\label{b75}\xymatrix{
\ih_i(Y)\ar[r]^{u^*}\ar[d]_{s^!_\sigma=e^{op}_\sigma(E,s)} & \ih_{i-2l}(Y')\ar[d]^{{s'}^!_{\sigma'}=e^{op}_{\sigma'}(E',s')}\\
H_{i-2r}(S)\ar[r]^{u_S^!} & H_{i-2r-2l}(S')
}\eeq
where $u^*$ is the pullback by Proposition \ref{b71} and the bottom horizontal arrow is $u_S^!=e^{op}(L|_S,\ell|_S)$.
\end{prop}

\begin{proof}
We let $\rho:\tY\to Y$ be a resolution of the degeneracy of $\sigma$, i.e. a proper morphism, isomorphic over $Y-\fZ$, such that the pullback of $\sigma$ by $\rho$ gives us a short exact sequence \eqref{b47}.
Let $\xi\in \ih_i(Y)$ and pick $\zeta\in H_i(\tY)$ and $\eta\in H_i(\fZ)$ satisfying \eqref{b30} and \eqref{b76}. By slightly abusing the notation, let us denote by $u^!$ the homomorphism obtained by the cap product with the pullback of $e(L,\ell)$ by any morphism to $S$. For instance, $u^!\eta=\eta\cap e(L|_\fZ,\ell|_\fZ)\in H_{i-2l}(\fZ').$ By \eqref{b73} and \eqref{b38}, we have
\beq\label{b78} \epsilon_{Y'}(u^*(\xi))=u^!(\epsilon_Y(\xi))=u^!\rho_*(\xi)+u^!\jmath_*(\eta) =\rho'_*(u^!\zeta)+\jmath'_*(u^!\eta)\eeq
where $\jmath':\fZ'\to Y'$ and $\rho':\tY'=\tY\times_YY'\to Y'$ are the restrictions of $\jmath$ and $\rho$ respectively. 
Hence we can use $$\zeta'=u^!\zeta \and \eta'=u^!\eta$$ to compute ${s'}^!_{\sigma'}(u^*(\xi))=e^{op}_{\sigma'}(E',s')(u^*(\xi)).$

For the $\zeta$ part, by \eqref{b80}, we have
\beq\label{b79}
\zeta'\cap e(F',\tilde{s}')\cap e(\sO_{\tX'}(\tS'),t_{\tS'}) = 
u^!\zeta\cap e(F',\tilde{s}')\cap e(\sO_{\tX'}(\tS'),t_{\tS'}) 
\eeq
\[ = \zeta\cap e(L|_{\tY},\ell|_{\tY})\cap e(F',\tilde{s}')\cap e(\sO_{\tX'}(\tS'),t_{\tS'})\]
\[ = \zeta\cap e(L|_{\tY}\oplus F,\ell|_{\tY}\oplus \tilde{s})\cap e(\sO_{\tX'}(\tS'),t_{\tS'}) \]
\[ = \zeta\cap e(F,\tilde{s})\cap e(L|_{\tX},\ell|_{\tX})\cap  e(\sO_{\tX'}(\tS'),t_{\tS'})\]
\[ = \zeta\cap e(F,\tilde{s})\cap e(L|_{\tX}\oplus \sO_{\tX}(\tS),\ell|_{\tX}\oplus t_{\tS})\]
\[ = \zeta\cap e(F,\tilde{s})\cap  e(\sO_{\tX}(\tS),t_{\tS})\cap e(L|_{\tS},\ell|_{\tS})\]
\[= u^!\left( \zeta\cap e(F,\tilde{s})\cap  e(\sO_{\tX}(\tS),t_{\tS}) \right).\]
By the projection formula \eqref{b38}, ${\rho_{S'}}_*\circ u^!=u^!\circ {\rho_{S}}_*$ and thus
\beq\label{b81}
{\rho_{S'}}_*\left( \zeta'\cap e(F',\tilde{s}')\cap e(\sO_{\tX'}(\tS'),t_{\tS'})\right)
=u^!\left[{\rho_S}_*\left( \zeta\cap e(\tE,\tilde{s})\cap  e(\sO_{\tX}(\tS),t_{\tS}) \right) \right].\eeq

For the $\eta$ part, by a similar computation, we have
\beq\label{b83}
\eta'\cap e(E'|_{\fZ'},s'|_{\fZ'})=u^!\eta \cap e(E'|_{\fZ'},s'|_{\fZ'}) =u^!\left( \eta\cap e(E|_\fZ,s|_\fZ) \right).\eeq

Combining \eqref{b35}, \eqref{b78}, \eqref{b79} and \eqref{b83}, we have 
\beq\label{b84} 
{s'}^!_{\sigma'}(u^*(\xi))=u^!s^!_\sigma(\xi)\eeq
as desired. 
\end{proof}

\begin{prop}\label{c17}
Under Assumption \ref{c5}, we further suppose that we have a morphism $\theta:Y'\to Y$ of \DM stacks 
which is placid and obtained by a fiber diagram
\[\xymatrix{
Y'\ar[r]^\theta\ar[d] & Y\ar[d]\\
M'\ar[r]^{\theta_M} & M}\] 
where $M'$ and $M$ are smooth varieties. 
As in Proposition \ref{c9}, let $E'$, $s'$, $\sigma'$, $X'$, $\fZ'$, $S'$, etc denote the pullbacks (by fiber products) of $E, s, \sigma, X, \fZ, S$, etc to $Y'$.  
If $\theta^*=\theta^!:H_*(Y)\to H_*(Y')$,
then 
\beq\label{c20} \theta^!\circ s_\sigma^!={s'}^!_{\sigma'}\circ \theta^*:\ih_i(Y)\lra H_{i-2r}(S').\eeq
\end{prop}
\begin{proof}
The proof follows from Lemma \ref{c3}, \eqref{c2} and \eqref{c30} similarly as in the proofs of Propositions \ref{c9} and \ref{b60}. We leave the detail to the reader. 
\end{proof}


The following is an immediate consequence of Propositions \ref{b60} and \ref{c17} that will be useful in the subsequent section. 
\begin{coro}\label{b90}
Under the assumptions of Proposition \ref{b60} (resp. \ref{c17}), we further assume that there are smooth morphisms
$q:Y\to Z$ and $q':Y'\to Z$ that fit into a commutative diagram
\beq\label{b86}\xymatrix{
Y'\ar[r]^{f}\ar[dr]_{q'} &Y\ar[d]^q\\
& Z
}\eeq
for an irreducible variety $Z$. Suppose $f$ is $u$ (resp. $\theta$) in Proposition \ref{b60} (resp. \ref{c17}).  
Since $q$ and $q'$ are smooth, all the arrows in \eqref{b86} are placid and thus we have a commutative diagram
\beq\label{b87}\xymatrix{
\ih_i(Z)\ar[d]^{q^*}\ar[dr]^{{q'}^*}\\
\ih_{i+2m}(Y) \ar[r]^{f^*}\ar[d]_{s^!_\sigma} & \ih_{i+2m'}(Y')\ar[d]^{{s'}^!_{\sigma'}}\\
H_{i+2m-2r}(S)\ar[r]^{f^!} & H_{i+2m'-2r}(S')
}\eeq
where $m=\dim_\CC Y-\dim_\CC Z$ and $m'=\dim_\CC Y'-\dim_\CC Z$.
\end{coro}

\begin{rema}
Under Assumption \ref{c5}, if we further assume that there is a smooth morphism $q:Y\to Z$ of relative dimension $m$, we have the composition 
\[ H_i(Z)\mapright{q^*} H_{i+2m}(Y)\mapright{s^!_\sigma} H_{i+2m-2r}(S). \]
This is the Borel-Moore homology version of the cosection localized virtual pullback constructed in \cite{CKL}.
\end{rema}

\subsection{A variation of Theorem \ref{b18}}\label{S3.v}
In this subsection, we prove a stronger result than Theorem \ref{b18} under a slightly stronger assumption that requires $E$, $s$ and $\sigma$ extend to a smooth \DM stack $W$ containing $Y$ as a closed substack. 

\begin{assu}\label{f21}
Let $W$ be a smooth \DM stack over $\CC$ and $s_W$ be a section of an algebraic vector bundle $E_W$ of rank $r$ over $W$,
whose zero locus is denoted by $X=s_W^{-1}(0)$. Let $\sigma_W:E_W\to \sO_W$ be a homomorphism
of coherent sheaves on $W$. Let $Y$ be the zero locus of the regular function $\sigma_W\circ s_W$.  
Let $S=X\times_W \fZ_W$ where $\fZ_W=\sigma_W^{-1}(0)$ denotes the zero locus of $\sigma_W$.
\end{assu}
By definition, we have closed substacks $S\subset X\subset Y\subset W$. 

Under this stronger assumption, the cosection localized Gysin map $s^!_\sigma$ is defined for all homology classes.  
\begin{theo}\label{f18} Under Assumption \ref{f21}, we have
a homomorphism $$s^!_\sigma:H_i(Y)\lra H_{i-2r}(S)$$ such that $\imath_*\circ s_\sigma^!=s^!$. 
\end{theo}
\begin{proof}
Let $\rho_1:W_1\to W$ be the blowup of $W$ along $\fZ_W$ and let $\fZ_{W,1}$ be the exceptional divisor.
Then the pullback $\sigma_{W,1}$ of $\sigma_W$ gives rise to a short exact sequence
$$0\lra E'_1\lra E_1=\rho_1^*E_W\mapright{\sigma_{W,1}} \sO_{W_1}(-\fZ_{W,1})\lra 0$$
for a subbundle $E'_1$. 
Next, let $\rho_2:\widetilde{W}\to W_1$ be a resolution of singularities. 
Then the above short exact sequence lifts to
$$0  \lra \widetilde{E}'_W\lra \widetilde{E}_W\mapright{\widetilde{\sigma}_{W}} \sO_{\widetilde{W}}(-\widetilde{\fZ}_{W})\lra 0$$
for a Cartier divisor $\widetilde{\fZ}_W$ of $\widetilde{W}$. 
Let $\rho_W=\rho_1\circ\rho_2:\widetilde{W}\to W$. 

As $W-\fZ_W$ is smooth, $\rho_W$ is an isomorphism over $W-\fZ_W$. 
Let $\widetilde{Y}=\widetilde{W}\times_WY$ and $\fZ=\fZ_W\cap Y$.
Let $\jmath:\fZ\to Y$ denote the inclusion.  
Let $\rho:\widetilde{Y}\to Y$ be the restriction of $\rho_W$ to $\widetilde{Y}$.  
Let $E=E_W|_Y$, $s=s_W|_Y$, $\sigma=\sigma_W|_Y$ and $\tfZ=\tfZ_W\times_{\widetilde{W}}\widetilde{Y}$.
By restricting the short exact sequence above, we have a short exact sequence
$$0\lra \widetilde{E}'\lra \widetilde{E}\mapright{\tilde{\sigma}} \sO_{\widetilde{Y}}(-\tfZ)\lra 0.$$
As $\sigma\circ s=0$ by the definition of $Y$, the pullback $\tilde{s}$ of $s$ to $\widetilde{Y}$ 
is a section of $\widetilde{E}'$.

By Lemma \ref{f19}, for any $\xi\in H_i({Y})$, there are $\zeta\in H_i(\widetilde{Y})$ and $\eta\in H_i(\fZ)$ such that
\beq\label{f23} \xi=\rho_*\zeta+\jmath_*\eta \and \xi|_{Y-\fZ}=\zeta|_{Y-\fZ}.\eeq
Now we can repeat all the arguments in \S\ref{S3.1} and \S\ref{S3.3}, using \eqref{f23} instead of \eqref{b30}. 
We thus obtain the homomorphism $s^!_\sigma$ in the theorem which does not depend on any choices involved in its definition. 
\end{proof}


\bigskip

\section{A construction of quantum singularity theory}\label{S4}

In this section, we provide a topological construction of cohomological field theories of singularities by our cosection localized Gysin map in Theorem \ref{b18}. The axioms for cohomological field theories will follow from the propositions in \S\ref{S3}.

\subsection{Setup}\label{S4.1}

Let $S$ be a smooth proper separated \DM stack over $\CC$. 
 Let
\beq\label{a1} M\mapright{\alpha}F \eeq
 be a complex of locally free sheaves on $S$. 
Let $\bp_M:M\to S$ denote the bundle projection. Then $\alpha$ induces a section $s_M$ of $E_M=\bp_M^*F$ over $M$. Let $X$ denote the zero locus of $s_M$.

Let $B=\mathbb{A}_\CC^m$ be an affine space equipped with a quasi-homogeneous nondegenerate polynomial $\underline{w} $. 
Let $\bq_M:M\to B$ be a smooth morphism
and $\sigma_M:E_M\to \sO_M$ be a homomorphism such that 
\beq\label{a2} \sigma_M\circ s_M=\underline{w} \circ\bq_M\and X\cap\sigma_M^{-1}(0)_\redd=S\eeq
where $\sigma_M^{-1}(0)_\redd$ denotes the support of the closed substack of $M$ defined by the image of $\sigma_M$ in $\sO_M$. 
The following diagram summarizes our setup so far.
\beq\label{a3}
\xymatrix{
&& E_M\ar[d]\ar[r]^{\sigma_M} & \sO_M\ar@{=}[r] & M\times\CC\ar[r] & \CC\\
X\ar@{=}[r] & s^{-1}_M(0)\ar[dr]_{\bp_X}\ar@{^(->}[r]^\imath & M\ar@/^1.0pc/[u]^{s_M}\ar[urrr]_{\underline{w} \circ\bq_M}\ar[dr]^{\bq_M}\ar[d]^{\bp_M}\\
&&S&B\ar[r]^{\underline{w} } &\CC
}\eeq
By the first equality in \eqref{a2}, we have the \emph{residue condition}
\beq\label{a7} \bq_M(X)\subset \underline{w} ^{-1}(0)=:Z.\eeq

Let $Y=Z\times_{B}M$ so that we have a fiber diagram
\beq\label{d1}\xymatrix{
X\ar[dr]_{\bq_X}\ar@{^(->}[r]^\imath & Y\ar[r] \ar[d]^{\bq_Y} &M\ar[d]^{\bq_M}\\
&Z \ar@{^(->}[r] & B
}\eeq
where $\bq_Y$ is smooth since $\bq_M$ is smooth. We denote the restriction of $E_M$ (resp. $\sigma_M$, resp. $s_M$) to $Y$ by $E$ (resp. $\sigma$, resp. $s$) so that  $X=s^{-1}(0)$. By \eqref{a7}, we have $\sigma\circ s=0$.

Since $Z$ has only an isolated hypersurface singularity, if $\dim_\CC Z=m-1\ge 2$, $Z$ is a normal affine variety and hence $Y$ is a normal \DM stack as $\bq_Y$ is smooth. 
When $m=2$, we replace $Z$ by its normalization. Since the intersection homology remains the same under normalization, all the arguments in this paper go through (cf. Remark \ref{e50}). 
The case $m=1$ in our FJRW setup (cf. \S\ref{S4.5}) occurs only when $N=1$ and there is only one broad marking. In this case, the section $x$ must vanish  
by the residue theorem. Hence the cosection localization of \cite{KLc} applies as in \cite{CLL}. 
With this preparation, we can now apply Theorem \ref{b18}.

By Theorem \ref{b18}, we thus have the cosection localized Gysin map
\beq\label{d2} s_\sigma^!=e^{op}_\sigma(E,s):\ih_i(Y)\lra H_{i-2r_1}(S)\eeq
where $r_1$ is the rank of $E$. Moreover since $\bq_Y$ is smooth, we have the pullback homomorphism
\beq\label{d3} \bq_Y^*:\ih_i(Z)\lra \ih_{i+2D_{\bq_Y}}(Y)\eeq
where $D_{\bq_Y}$ is the relative complex dimension of $\bq_Y$ which equals 
\beq\label{d4}
D_{\bq_Y}=\dim_\CC M-\dim_\CC B=\dim_\CC S+r_0-\dim_\CC B
\eeq
where $r_0$ denotes the rank of the vector bundle $\bp_M:M\to S$.
Composing \eqref{d2} and \eqref{d3}, we obtain 
\beq\label{d5}
\Phi:\ih_i(Z)\lra H_{2\dim_\CC S-2D+i-2\dim_\CC B}(S)
\eeq
where 
\beq\label{d6}
D=-r_0+r_1=-\mathrm{rank}(M)+\mathrm{rank}(E)=-\mathrm{rank}(M)+\mathrm{rank}(F)
\eeq
following the notation of \cite[(52)]{FJR}. If \eqref{a1} is a resolution of a sheaf complex $\mathcal{U}$, then $D$ is minus the rank of the complex $\mathcal{U}$.

\bigskip

Suppose we have another complex 
\beq\label{e36} M'\mapright{\alpha'} F'\eeq 
satisfying all the assumptions for \eqref{a1} like \eqref{a2}
so that \eqref{e36} also gives us diagrams like \eqref{a3} and \eqref{d1}, where $M$, $Y$ and $X$ are replaced by $M'$, $Y'$ and $X'={s'}^{-1}(0)$. (We use the same $S$, $Z$ and $B$.) By the recipe above, we then have a vector bundle $E'$ on $Y'$ equipped with a section $s'$ and a cosection $\sigma'$ as well as the cosection localized Gysin map ${s'}^!_{\sigma'}:\ih_*(Y')\to H_*(S)$.  We can compare $s_\sigma^!$ and ${s'}^!_{\sigma'}$ in the following two cases. 
\begin{prop}\label{e34}
Suppose \eqref{a1} and \eqref{e36} fit into a commutative diagram of exact sequences of locally free sheaves
\[\xymatrix{
0\ar[r] & M'\ar[r]\ar[d]_{\alpha'} & M\ar[r]\ar[d]^{\alpha} & K\ar[r]\ar@{=}[d] &0\\
0\ar[r] & F'\ar[r] & F\ar[r] &K\ar[r] &0.
}\]
If $\sigma'=\sigma|_{E'}$ and $\sigma^{-1}(0)\times_YY'={\sigma'}^{-1}(0)$, then ${s'}^!_{\sigma'}\circ\imath^*=s^!_\sigma$ where $\imath:Y'\to Y$ denotes the placid inclusion induced from the inclusion map  $M'\to M$.
\end{prop}

\begin{proof}
Suppose $\sigma=0$ so that we only have the ordinary topological Gysin maps. We can continuously deform $(M,s)$ to $(M'\oplus K,s'\oplus 1_K)$. The proposition is a direct consequence of \cite[IX.9.3]{Iver} for the direct sum case. Since $s^!$ is defined over $\QQ$, $s^!$ remains constant under the deformation. So the proposition holds when $\sigma=0$.

Under our assumption, if we pick a $\sigma$-regularizing morphism $\tilde{Y}\to Y$, then the fiber product $\tilde{Y}'=\tilde{Y}\times_YY'\to Y'$ is $\sigma'$-regularizing. If $\zeta$ is a lift of $\xi$ to $H_*(\tilde{Y})$, $\imath^*\zeta$ is a lift of $\imath^*\xi$ to $H_*(\tilde{Y}')$. The proposition easily follows from the definition of $s^!_\sigma$ in Theorem \ref{b18} and the case for $\sigma=0$. 
\end{proof}

\begin{prop}\label{e35}
Suppose \eqref{a1} and \eqref{e36} fit into a commutative diagram of exact sequences of locally free sheaves
\[\xymatrix{
0\ar[r] &K\ar[r] \ar@{=}[d]& M'\ar[r]\ar[d]_{\alpha'} & M\ar[r]\ar[d]^{\alpha}  &0\\
0\ar[r] &K\ar[r] & F'\ar[r] & F\ar[r] &0.
}\]
If $\sigma'=\sigma|_{E'}$, then ${s'}^!_{\sigma'}\circ f^*=s^!_\sigma$ where $f:Y'\to Y$ denotes the smooth morphism induced from the surjection  $M'\to M$.
\end{prop}
\begin{proof}
The proof is similar to Proposition \ref{e34} and we omit it. 
\end{proof}

\subsection{State space}\label{S4.2}
In this subsection, we recall the basic setup for the Fan-Jarvis-Ruan-Witten theory from \cite{FJR} and define our state space. 
 
 Let $w:\CC^N\to \CC$ be a quasi-homogeneous polynomial with 
\beq\label{d7}
w(t^{d_1}x_1,\cdots,t^{d_N}x_N)=t^d \cdot w(x_1,\cdots,x_N),\quad \exists d_i, d\in \ZZ_{>0}\eeq
which is nondegenerate, i.e.
\begin{enumerate}
\item $w$ has no monomial of the form $x_ix_j$ for $i\ne j$ and
\item the hypersurface 
$$Q_w=\PP w^{-1}(0)\subset \PP^{N-1}_{d_1,\cdots,d_N}$$
defined by $w$ is nonsingular.
\end{enumerate} 
We assume $d>0$ is minimal possible and $d>d_i$. Let $q_i=d_i/d$.  
The second condition above implies that the hypersurface $w^{-1}(0)\subset \CC^N$ has only an isolated singularity at $0$ and that $q_i\le \frac12$ are uniquely determined by $w$. 

Writing $w=\sum_{k=1}^\nu c_kw_k$ with $c_k\in \CC^*$ and $w_k$ distinct monomials, we obtain a homomorphism
\beq\label{d25} (w_1,\cdots, w_\nu): (\CC^*)^N\lra (\CC^*)^\nu \eeq
whose kernel is the symmetry group of $w$ defined by
\beq\label{d8} G_w=\{(\lambda_1\cdots,\lambda_N)\in (\CC^*)^N\,|\,w(\lambda_1 x_1,\cdots,\lambda_N x_N)=w(x_1,\cdots,x_N)\}\eeq
which is always finite under our assumption. 
By \eqref{d7}, $G_w$ has an element
$$J_w=(e^{2\pi iq_1},\cdots, e^{2\pi iq_N})\in G_w.$$
We fix a subgroup $G$  of $G_w$ containing $J_w$. The input datum for the FJRW quantum singularity theory is the pair $(w,G)$, sometimes denoted $w/G$.

Consider the fiber product 
\[\xymatrix{
\hat{G}_w\ar[r] \ar[d] &\CC^*\ar[d]\\
(\CC^*)^N\ar[r] & (\CC^*)^\nu
}\]
where the bottom horizontal is \eqref{d25} and the right vertical is the diagonal embedding. 
By \eqref{d7}, the homomorphism $$\CC^*\to  (\CC^*)^N,\quad t\mapsto (t^{d_1},\cdots,t^{d_N})$$
factors through $\hat{G}_w$ and we thus have a surjective homomorphism $G_w\times \CC^*\to \hat{G}_w$ whose kernel is $\mu_d=\{z\in \CC\,|\, z^d=1\}.$
Hence the subgroup $G$ of $G_w$ defines a subgroup 
$$\hat{G}=G\times\CC^*/\mu_d$$ of $\hat{G}_w$ that fits into the exact sequence 
\beq\label{d26} 1\lra G \lra \hat{G} \mapright{\chi} \CC^*\lra 1.\eeq
Note that for ${\lambda}\in \hat{G}\subset (\CC^*)^N$, 
\beq\label{d31} w(\ulambda\cdot x)=\chi(\ulambda)w(x).\eeq 
In particular, $w(\ulambda\cdot x)=-w(x)$ if and only if $\ulambda\in\chi^{-1}(-1)$.

For the singularity $w/G$, the \emph{state space} is defined as
\beq\label{d9} \cH=\bigoplus_{\gamma\in G}\cH_\gamma,\quad \cH_\gamma=H^{N_\gamma}(\CC^{N_\gamma},w_\gamma^\infty)^G\eeq
where $\CC^{N_\gamma}$ denotes the $\gamma$-fixed subspace of $\CC^N$, $w_\gamma$ is the restriction of $w$ to $\CC^{N_\gamma}$ and  $w_\gamma^\infty=(\mathrm{Re}(w_\gamma))^{-1}(a,\infty)$ for $a>\!>0$. When $\gamma=(e^{2\pi iq_1},\cdots, e^{2\pi iq_N})$,  $\cH_\gamma=\CC$ since $N_\gamma=0$. The constant function $1$ in $\cH_{(e^{2\pi iq_1},\cdots, e^{2\pi iq_N})}$ is denoted by $\mathbf{1}$.

\begin{lemm}\label{g0} We have a natural isomorphism
$$H^{N_\gamma}(\CC^{N_\gamma},w_\gamma^\infty)^G\cong H^{N_\gamma-2}_{\mathrm{prim}}(Q_{w_\gamma})^G.$$
\end{lemm}
\begin{proof} To keep the notation simple, we delete $\gamma$ in the proof. Since $G$ is finite, 
\beq\label{g1} H^N(\CC^N,w^\infty)^G=H^N(\CC^N/G,w^\infty).\eeq
Since $\CC^N/G\cong (\CC^N\times\CC^*)/\hat{G}$ where $\hat{G}$ acts on the last component by weight $\chi^{-1}$ (cf. \S\ref{S1}), we have
\beq\label{g2} H^N(\CC^N/G,w^\infty)\cong H^N( (\CC^N\times\CC^*)/\hat{G},\mathbf{w}^\infty)
\cong H^N_{\hat{G}}( \CC^N\times\CC^*,\mathbf{w}^\infty)\eeq
where $\mathbf{w}(z,t)=t\cdot w(z)$ for $(z,t)\in \CC^N\times \CC^*$. 
The last isomorphism to equivariant cohomology follows from the fact that $\hat{G}$ acts quasi-freely. 
By the well-known argument of Atiyah-Bott, as the weight on the last component $\CC$ is nontrivial, the Gysin sequence 
\beq\label{g3} \cdots \lra H^{N-2}_{\hat{G}}( \CC^N)\lra H^N_{\hat{G}}( \CC^N\times\CC,\mathbf{w}^\infty)\lra H^N_{\hat{G}}( \CC^N\times\CC^*,\mathbf{w}^\infty)\lra \cdots\eeq
splits into short exact sequences
\beq\label{g4}
0\lra H^{N-2}_{\hat{G}}( \CC^N)\lra H^N_{\hat{G}}( \CC^N\times\CC,\mathbf{w}^\infty)\lra H^N_{\hat{G}}( \CC^N\times\CC^*,\mathbf{w}^\infty)\lra  0.
\eeq
On the other hand, by Gysin sequence, we have an isomorphism
\beq\label{g5}
H^N_{\hat{G}}( \CC^N\times\CC,\mathbf{w}^\infty)\cong H^N_{\hat{G}}( (\CC^N-0)\times\CC,\mathbf{w}^\infty)
\cong H^N(\cL(\chi^{-1}), w^\infty)^G\eeq
where $\cL(\chi^{-1})=(\CC^N-0)\times\CC/\CC^*$ is the line bundle over the weighted projective space
$\PP^{N-1}_{d_1,\cdots,d_N}$ defined by the weight $\chi^{-1}$. It is straightforward to check that the critical locus in $\cL(\chi^{-1})$ is precisely $\PP w^{-1}(0)=Q_w$, all of whose points are nondegenerate, i.e. there is an analytic local coordinate system at each point such that $w=x^2+y^2$ where $x,y$ are the normal coordinates of $Q_w$ in $\cL(\chi^{-1})$. Hence we have
\beq\label{g6}
H^N(\cL(\chi^{-1}), w^\infty)^G\cong H^{N-2}(Q_w)^G.\eeq
Combining \eqref{g4}, \eqref{g5} and \eqref{g6}, we find that 
\beq\label{g7}
H^N_{\hat{G}}( \CC^N\times\CC^*,\mathbf{w}^\infty)\cong H^{N-2}_{\mathrm{prim}}(Q_w)^G.
\eeq 
The lemma follows from \eqref{g1}, \eqref{g2} and \eqref{g7}.
\end{proof}

Consequently we have 
\beq\label{d10}
\cH_\gamma=H^{N_\gamma}(\CC^{N_\gamma},w_\gamma^\infty)^G \cong H^{N_\gamma-2}_{\mathrm{prim}}(Q_{w_\gamma})^G\cong \ih_{N_\gamma}(w_\gamma^{-1}(0))^G
\eeq
by \eqref{d11} and 
\beq\label{d12}
\cH=\bigoplus_{\gamma\in G}\cH_\gamma = \bigoplus_{\gamma\in G}\ih_{N_\gamma}(w_\gamma^{-1}(0))^G.\eeq

Next we define a nondegenerate pairing on $\cH$. 
Since the restriction $w_\gamma$ of $w$ is also nondegenerate \cite[Lemma 2.1.10]{FJR}, $Q_{w_\gamma}$
is a nonsingular projective variety and we have the perfect intersection pairing on $H^{N_\gamma-2}_{\mathrm{prim}}(Q_{w_\gamma})$  which induces perfect pairings on $\ih_{N_\gamma}(w_\gamma^{-1}(0))$ and $\cH_\gamma=\ih_{N_\gamma}(w_\gamma^{-1}(0))^G$ as $G$ is an automorphism group of $Q_{w_\gamma}$.  
Since $\CC^{N_{\gamma}}=\CC^{N_{\gamma^{-1}}}$ and $w_\gamma=w_{\gamma^{-1}}$, we have 
$\cH_\gamma=\cH_{\gamma^{-1}}$. Using this identification and the perfect pairing on each $\cH_\gamma$, we have a perfect pairing
\beq\label{d13}
\langle \cdot, \cdot \rangle:\cH_\gamma\otimes \cH_{\gamma^{-1}}\lra \CC.
\eeq
Summig them up for $\gamma\in G$, we obtain a perfect pairing $\langle,\rangle$ on the state space $\cH$.

\begin{exam}\label{d14}
Let $w:\CC^5\to\CC$ be $w(x_1,\cdots,x_5)=\sum_ix_i^5$ with $q_1=\cdots=q_5=\frac15$. Let $G=\mu_5$ be the cyclic group of order 5 generated by $e^{\frac{2\pi i}{5}}$ acting diagonally on $\CC^5$. 
Let $\chi:\CC^*\to \CC^*, t\mapsto t^5$ so that $\chi^{-1}(1)=G$. Then $Q_w$ is the quintic Fermat Calabi-Yau 3-fold. For $\gamma\ne 1$, $N_\gamma=0$ and $\cH_\gamma=\CC$. For $\gamma=1$, we have $\cH_\gamma=H^3(Q_w)=\ih_5(w^{-1}(0))$ as $G$ acts trivially on the cohomology. Hence the state space is the orthogonal sum
$$\cH=\CC^4_{e_1,\cdots, e_4}\oplus \ih_5(w^{-1}(0))\cong H^{even}(Q_w)\oplus H^3(Q_w)=H^*(Q_w)$$
and the perfect pairing on $\cH$ is the usual intersection pairing on $H^*(Q_w).$ 
\end{exam}

Consider the Thom-Sebastiani sum $$w_\gamma\boxplus w_{\gamma'}: \CC^{N_\gamma}\oplus \CC^{N_{\gamma'}}\to \CC,\quad (w_\gamma\boxplus w_{\gamma'})(x,y)=w_\gamma(x)+w_{\gamma'}(y).$$ 
By \cite{Massey}, we have an isomorphism of vanishing cohomology 
\beq\label{d17} H^{N_\gamma-2}_{\mathrm{van}}(Q_{w_\gamma})\otimes H^{N_{\gamma'}-2}_{\mathrm{van}}(Q_{w_{\gamma'}})\cong H^{N_\gamma+N_{\gamma'}-2}_{\mathrm{van}}(Q_{w_\gamma\boxplus w_{\gamma'}})
\eeq
whose Poincar\'e dual sends $(\xi,\eta)\in H^{N_\gamma-2}_{\mathrm{van}}(Q_{w_\gamma})\otimes H^{N_{\gamma'}-2}_{\mathrm{van}}(Q_{w_{\gamma'}})$ to the join $\xi *\eta\in H^{N_\gamma+N_{\gamma'}-2}_{\mathrm{van}}(Q_{w_\gamma\boxplus w_{\gamma'}})$, i.e. the union of lines joining a point in $\xi$ with a point in $\eta$. 
By \eqref{d17}, we have an isomorphism 
\beq\label{d18} \ih_{N_\gamma}(w_\gamma^{-1}(0))\otimes \ih_{N_{\gamma'}}(w_{\gamma'}^{-1}(0))\cong \ih_{N_\gamma+N_{\gamma'}}((w_\gamma\boxplus w_{\gamma'})^{-1}(0)). \eeq
If we apply \eqref{b21}, \eqref{d18} gives us \eqref{b17}. 
We therefore have
\beq\label{d19} \cH_\gamma \otimes \cH_{\gamma'} \cong  \ih_{N_\gamma}(w_\gamma^{-1}(0))^G\otimes \ih_{N_{\gamma'}}(w_{\gamma'}^{-1}(0))^G\cong \ih_{N_\gamma+N_{\gamma'}}((w_\gamma\boxplus w_{\gamma'})^{-1}(0))^{G\times G}.\eeq

Recall from \eqref{d31} that if $\ulambda\in \hat{G}$ belongs to the coset $\chi^{-1}(-1)$, $w_\gamma(\ulambda\cdot x)=-w_\gamma(x)$ and vice versa. For $\ulambda\in \chi^{-1}(-1)$, we have an embedding
\beq\label{d15} \CC^{N_\gamma}\hookrightarrow (w_\gamma\boxplus w_\gamma)^{-1}(0),\quad x\mapsto (x,\ulambda\cdot x)\eeq
whose image is denoted by $\Delta_\ulambda$. As the origin is the only singular point, we find that $\Delta_\ulambda$ is an allowable geometric chain for the middle perversity and defines a class 
\beq\label{d16}
[\Delta_\ulambda]\in \ih_{2N_\gamma}((w_\gamma\boxplus w_\gamma)^{-1}(0))\cong \ih_{N_\gamma}(w_\gamma^{-1}(0))\otimes \ih_{N_\gamma}(w_\gamma^{-1}(0)). \eeq
Since $\Delta=\bigcup_{\ulambda\in \chi^{-1}(-1)} \Delta_\ulambda$ is obviously $G$-invariant, its  homology class lies in the $G\times G$-invariant part, i.e. 
\small
$$[\Delta]=\sum_{\ulambda\in \chi^{-1}(-1)} [\Delta_\ulambda]\ \in \ \ih_{2N_\gamma}((w_\gamma\boxplus w_\gamma)^{-1}(0))^{G\times G}\cong \ih_{N_\gamma}(w_\gamma^{-1}(0))^G\otimes \ih_{N_\gamma}(w_\gamma^{-1}(0))^G.$$
\normalsize
Choosing a basis $\{e_i\}$ of $\cH_\gamma \cong  \ih_{N_\gamma}(w_\gamma^{-1}(0))^G$, we can write 
$$[\Delta]=\sum_{i,j}a_{ij}e_i\otimes e_j\in \ih_{N_\gamma}(w_\gamma^{-1}(0))^G\otimes \ih_{N_\gamma}(w_\gamma^{-1}(0))^G.$$
Since $e_l$ are $G$-invariant, it is now straightforward to check that the intersection pairing $[\Delta]\cdot (e_k\otimes e_l)$ is $|G|\,c_{kl}$ where $c_{kl}=\langle e_k,e_l\rangle$. From this, by linear algebra, we obtain that
\beq\label{d21} [\Delta]=|G|\,\sum_{i,j} c^{ij}e_i\otimes e_j\eeq
where $(c^{ij})$ is the inverse matrix of $(c_{ij})$. 


\subsection{Spin curves}\label{S4.5}
In this subsection, given the input data $w/G$, we associate the moduli space of $G$-spin curves.  

A \emph{twisted curve} (orbicurve) with markings refers to a proper \DM stack $C$ with markings $p_1,\cdots, p_n$ in $C$, such that 
\begin{enumerate}
\item if we let $\rho:C\to |C|$ denote the coarse moduli space, $|C|$ is a projective curve with at worst nodal singularities and $\rho(p_i)$ are smooth points in $|C|$;
\item $\rho$ is an isomorphism away from nodes or maked points;
\item each marking is locally $\CC/\mu_{l}$ for some $l>0$ where $\mu_l=\{z\in \CC\,|\, z^l=1\}$ acts by multiplication;
\item each node is locally $\{xy=0\}/\mu_l$ for some $l>0$ where the action of $z\in \mu_l$ is $(x,y)\mapsto (zx,z^{-1}y)$.
\end{enumerate}
We let $\omega_C\ulog=\rho^*\omega_{|C|}\ulog=\rho^*\omega_{|C|}(|p_1|+\cdots+|p_n|)$ where $|p_j|=\rho(p_j)$. For each marking and node, we fix a generator of the cyclic stabilizer group of the point. 

A \emph{$G$-spin curve} is a twisted curve $C$ with markings $p_1,\cdots, p_n$, together with a principal $\hat{G}$-bundle $P$ on $C$ and an isomorphism
\[ \varphi:\chi_*P\cong P(\omega_C\ulog)\] 
of principal $\CC^*$-bundles where $P(\omega_C\ulog)$ denotes the principal  $\CC^*$-bundle of the line bundle $\omega_C\ulog$ and $\chi_*P$ denotes the principal $\CC^*$ bundle obtained by applying $\chi$ to the fibers of $P$. 
Moreover, the inclusion map $\hat{G}\to (\CC^*)^N$ and $P$ induce the principal $(\CC^*)^N$-bundle $P\times_{\hat{G}}(\CC^*)^N$ which gives us an $N$-tuple of line bundles $(L_1,\cdots, L_N)$. 
For each marking $p_j$ of a twisted curve $C$,  the fixed generator of the stabilizer group acts on the fiber $L_i|_{p_j}$ of $L_i$ over $p_j$ by multiplication by a constant $\gamma_{ij}$. 
Let 
\beq\label{d28} \gamma_j=(\gamma_{ij})_{1\le i\le N}\in G, \quad \ugamma=(\gamma_1,\cdots,\gamma_n)\in G^n.\eeq
Likewise, the fixed generator of the stabilizer group of a node $p$ acts on the fibers of $L_i$ and gives us an element $\gamma_p$ of $G$. We call $\ugamma$ the \emph{type} of the spin curve. 
  
The spin curve $(C,p_j, L_i,\varphi)$ is called \emph{stable} if $(|C|,|p_1|,\cdots, |p_n|)$ is a stable curve and the homomorphism from the stabilizer group of a marking $p_j$ (resp. a node $p$) into $G$ that sends the generator to $\gamma_j$ (resp. $\gamma_p$) is injective.

\begin{theo}\label{d29}\cite[Theorem 2.2.6]{FJR} \cite[Proposition 3.2.6]{PV}
The stack $S_{g,n}$ of stable $G$-spin curves is a smooth proper \DM stack with projective coarse moduli.
The morphism $S_{g,n}\to \overline{M}_{g,n}$ sending a spin curve $(C,p_j,L_i,\varphi)$ to $(|C|, |p_j|)$ is flat proper and quasi-finite.  
\end{theo}

A \emph{rigidification} of a spin curve at a marking $p_j$ is an isomorphism 
$$\psi_j:L_1|_{p_j}\oplus \cdots \oplus L_N|_{p_j}\mapright{\cong} \CC^N/\langle \gamma_j\rangle,$$
where $\langle \gamma_j\rangle \le G$ is the subgroup generated by $\gamma_j$, 
such that 
$$w_k\circ \psi_j=\mathrm{res}_{p_j}\circ \varphi_k|_{p_j}$$ for each monomial $w_k$ of $w$ where $\varphi_k:\oplus_iL_i\to \omega_C\ulog$ is the map defined by $w_k$ and $\varphi$. 
Two different rigidifications $\psi=(\psi_j)$ and $\psi'=(\psi_j')$ are related by the action of $\prod_{j=1}^n G/\langle \gamma_j\rangle.$ 
Hence the moduli stack $S_{g,n}\urig$ of stable $G$-spin curves with rigidification is a proper smooth \DM stack which is an \'etale cover over $S_{g,n}$. 
If we denote the moduli stack of spin curves of type $\ugamma$ with rigidification by $S\urig_{g,\ugamma}$, we have the disjoint union $$S\urig_{g,n}=\sqcup_\ugamma S\urig_{g,\ugamma}.$$


From now on, let $S=S\urig_{g,\ugamma}.$ 
Let $\pi:\cC\to S$ denote the universal curve and $\cL_i$ be the universal line bundle. 
By \cite[\S4.2]{PV}, 
there is a resolution 
\beq\label{d34} R\pi_*(\oplus_{i=1}^N \cL_i)\cong [M\mapright{\alpha} F]\eeq
by locally free sheaves $M$ and $F$ over $S$ constructed as follows: 
We fix an injective homomorphism $\oplus_i\cL_i\to P$ of locally free sheaves with $R^1\pi_*P=0$
such that $P/(\oplus_i\cL_i)$ is locally free; 
we let $\Sigma$ be the union of markings, and form the obvious injection 
$\oplus_i\cL_i\to P\oplus(\oplus _i\cL_i)|_{\Sigma}=P'$. 
Let $Q=P'/(\oplus_i\cL_i)$. Then 
the two-term complex $[M\to F]$ in \eqref{d34} is defined by the fiber product
\[\xymatrix{
M\ar[r]\ar[d] & F\ar[d]\\
\pi_*P'\ar[r] & \pi_*Q
}\]
for a surjective homomorphism $F\to \pi_*Q$ from a sufficiently negative vector bundle $F$. The last surjection was added 
as it is necessary for the construction of a cosection (cf. \cite[\S4.2]{PV}).

The projection $P'\to (\oplus _i\cL_i)|_{\Sigma}$ together with rigidification isomorphisms give us a smooth morphism
$$\bq_M:M\lra \pi_*P'\lra \pi_*\bl (\oplus _{i}\cL_i)|_{\Sigma}\br \lra B=\prod_{j=1}^n\CC^{N_{\gamma_j}}$$ 
where $\CC^{N_{\gamma_j}}$ denotes the $\gamma_j$-fixed subspace of $\CC^N$.
By construction, $\bq_M$ is smooth on each fiber over $S$.

The restriction $w_{\gamma_j}$ of $w$ to $\CC^{N_{\gamma_j}}$ gives us 
$${w}_{\ugamma} =w_{\gamma_1}\boxplus \cdots \boxplus w_{\gamma_n}.$$
Let $E_M=\bp_M^*F$ and $s_M$ be the section of $E_M$ defined by $\alpha$ where $\bp_M:M\to S$ is the bundle projection. 
Let $$X\urig_{g,n}=\sqcup_\ugamma X\urig_{g,\ugamma}$$ be the moduli space of spin curves $(C,p_j, L_i,\varphi)$ with rigidification $\psi$ and sections $(x_1,\cdots,x_N)\in \oplus_i H^0(L_i)$ of the line bundles $L_i$. Then $X\urig_{g,\ugamma}=\zero(s_M)$ by construction. 
Moreover, by \cite[\S4.2]{PV}, there is a homomorphism 
$$\sigma_M:E_M\to \sO_M$$ satisfying
\eqref{a2}. 
By the construction of $\sigma_M$ in \cite[\S4.2]{PV}, we have \eqref{a2}. 
 
In summary, with $S=S\urig_{g,\ugamma}$, $X=X\urig_{g,\ugamma}$, $B_\ugamma=\prod_{j=1}^n\CC^{N_{\gamma_j}}$ and $\underline{w}={w}_\ugamma$, all the assumptions in \S\ref{S4.1} are satisfied, and hence from \eqref{d5} we have the cosection localized Gysin map 
\beq\label{d35} \bigotimes_j\cH_{\gamma_j}\cong \ih_{\sum_j N_{\gamma_j}}({w}_\ugamma ^{-1}(0))^{G^n}\mapright{\Phi} H_{2\mathrm{vd}(X)-\sum N_{\gamma_j}}(S_{g,\ugamma}\urig)\eeq
where 
$$\mathrm{vd}(X)=3g-3+n+\sum_i\chi(L_i)$$ 
is the virtual dimension of $X$ and the isomorphism in \eqref{d35} is the Thom--Sebastiani isomorphism. 
Note that the dimension $2\mathrm{vd}(X)-\sum N_{\gamma_j}$ matches the degree calculation in \cite[Theorem 4.1.1]{FJR}.

\subsection{Cohomological field theory}\label{S4.3}
Let $\cH$ be defined by \eqref{d12} equipped with perfect pairing \eqref{d13}. 
A \emph{cohomological field theory} with a unit for the state space $\cH$ consists of homomorphisms
\beq\label{d36} \Omega_{g,n}:\cH^{\otimes n}\lra H^*(\rmgn)\eeq
in the stable range $2g-2+n>0$ satisfying the following axioms:
\begin{enumerate}
\item $\Omega_{g,n}$ is $S_n$-equivariant where the symmetric group $S_n$ acts on $\rmgn$ by permuting the markings and on $\cH^{\otimes n}$ by permuting the factors; 
\item let $u:\overline{M}_{g-1,n+2}\to \overline{M}_{g,n}$ be the gluing of the last two markings, then for a basis $\{e_k\}$ of $\cH$ and $c_{kl}=\langle e_k,e_l\rangle$ with $(c^{kl})=(c_{kl})^{-1}$, we have
\beq\label{d37} u^*\Omega_{g,n}(v_1,\cdots,v_n)=\sum_{k,l} c^{kl}\Omega_{g-1,n+2}(v_1,\cdots,v_n, e_k,e_l),\quad v_i\in \cH\eeq 
where $u^*:H^*(\overline{M}_{g,n})\to H^*(\overline{M}_{g-1,n+2})$ is the pullback by $u$;
\item let $u:\overline{M}_{g_1,n_1+1}\times \overline{M}_{g_2,n_2+1}\to \overline{M}_{g,n}$ with $g=g_1+g_2$ and $n=n_1+n_2$ be the gluing of the last markings, then for $v_i\in \cH$,
\beq\label{d38} u^*\Omega_{g,n}(v_1,\cdots,v_n)=\sum_{k,l} c^{kl}\Omega_{g_1,n_1+1}(v_1,\cdots,v_{n_1}, e_k)\otimes\qquad\qquad  \eeq
$$\qquad\qquad\qquad\qquad\qquad\qquad\qquad\qquad\otimes \Omega_{g_2,n_2+1}(v_{n_1+1},\cdots,v_{n}, e_l)$$
in $H^*(\overline{M}_{g_1,n_1+1})\otimes H^*(\overline{M}_{g_2,n_2+1})$;
\item let $\theta:\overline{M}_{g,n+1}\to \rmgn$ be the morphism forgetting the last marking, then we have 
\beq\label{d39} \Omega_{g,n+1}(v_1,\cdots,v_n,\mathbf{1})=\theta^*\Omega_{g,n}(v_1,\cdots,v_n),\quad \forall v_i\in \cH ;\eeq
\item $\Omega_{0,3}(v_1,v_2,\mathbf{1})=\langle v_1,v_2\rangle$ for $v_i\in \cH$. 
\end{enumerate}

We have a forgetful morphism 
$$\mathrm{st}:S=S_{g,\ugamma}\urig\lra \rmgn$$
whose pushforward is denoted by $\mathrm{st}_*:H_*(S)\to H_*(\rmgn)\cong H^*(\rmgn)$.
Composing \eqref{d35} with 
$$\frac{(-1)^D}{\deg \, \mathrm{st}} \mathrm{st}_*:H_*(S)\lra H_*(\overline{M}_{g,n})$$ 
with  $D=-\sum_i\chi(L_i)$, we obtain
\beq\label{a18} \Psi_\ugamma:\bigotimes_{1\le j\le n}\cH_{\gamma_j}\mapright{\Phi} H_*(S_{g,\ugamma}\urig)\lra H_*(\overline{M}_{g,n})\cong H^*(\overline{M}_{g,n}).\eeq
Since $\cH=\bigoplus_{\gamma\in G}\cH_\gamma$, by taking the sum of $\Psi_\ugamma$ over $\ugamma\in G^n$, we obtain 
\beq\label{d40}
\Omega_{g,n}=\sum_\ugamma\Psi_\ugamma:\cH^{\otimes n}\lra H^*(\rmgn).\eeq

\begin{theo}\label{a19} The homomorphisms $\{\Omega_{g,n}\}_{2g-2+n>0}$ in \eqref{d40} define a cohomological field theory with a unit for the state space $\cH$.
Moreover this cohomological field theory coincides with that in \cite[Theorem 4.2.2]{FJR} when we restrict $\cH$ to the narrow sector  $\bigoplus_{\ugamma: N_{\gamma_j}=0,\, \forall j}\bigotimes_j\cH_{\gamma_j}$. 
\end{theo}
\begin{proof} We will check the axioms for \eqref{d36} above. 
Indeed, (1) follows from the construction directly, and (5) is obvious. We will prove the remainder.

For (2), from \S\ref{S4.5}, we have a diagram
\beq\label{e1}\xymatrix{
S\urig_{g,\ugamma}\ar[d]_{\mathrm{st}} & Y=Y_{g,\ugamma}\ar@{^(->}[r]\ar[d]^{\bq_Y} \ar[l]_{\bp_Y} & M\ar[d]^{\bq_M}\\
\rmgn & Z_\ugamma=w_\ugamma^{-1}(0)\ar@{^(->}[r] & B_\ugamma=\prod_{j=1}^n\CC^{N_{\gamma_j}}
}\eeq
The gluing of the last two nodes of curves in $\overline{M}_{g-1,n+2}$ gives us the embedding 
$$u:\overline{M}_{g-1,n+2}\lra \rmgn.$$ 
Consider the commutative diagram
\small
\beq\label{e2}\xymatrix{
\sqcup_{\lambda\in G}Z_{(\ugamma,\lambda,\lambda^{-1})}\ar@{}[dr]^{\Box} & \hat{Z}_\ugamma \ar[l]\ar[r] & Z_\ugamma\ar@{=}[r] & Z_\ugamma\\
\sqcup_{\lambda\in G}X_{g-1,(\ugamma,\lambda,\lambda^{-1})}\ar[u]_{\tilde{\bq}_X}\ar[d] & \hat{X}\ar@{}[dr]^{\Box}\ar[l]_{\tau_X}\ar[r]^{\kappa_X}\ar[u]_{\hat{\bq}_X} \ar[d] & X'\ar@{}[dr]^{\Box}\ar[r]^{u_X} \ar[u]_{\bq'_X}\ar[d] & X_{g,\ugamma}\ar[d]^{\bp_X}\ar[u]_{\bq_X}\\
\sqcup_{\lambda\in G}S\urig_{g-1,(\ugamma,\lambda,\lambda^{-1})}\ar[d]_{\tilde{\mathrm{st}}} & \hat{S}\ar[l]_\tau\ar[r]^\kappa \ar[d]^{\hat{\mathrm{st}}} & S'\ar@{}[dr]^{\Box}\ar[r]\ar[d]^{\mathrm{st}'} & S\urig_{g,\ugamma}\ar[d]^{\mathrm{st}}\\
\overline{M}_{g-1,n+2}\ar@{=}[r]&\overline{M}_{g-1,n+2}\ar@{=}[r]&\overline{M}_{g-1,n+2}\ar[r]^u & \rmgn.
}\eeq
\normalsize
Here $\hat{S}= \sqcup_{\lambda\in G}S\urig_{g-1,(\ugamma,\lambda,\lambda^{-1})}\times \chi^{-1}(-1)$, $\hat{Z}_\ugamma=(\sqcup_{\lambda\in G} Z_\ugamma\times \CC^{N_\lambda})\times \chi^{-1}(-1)$  and $\tau$ is the projection. An object in $S\urig_{g-1,(\ugamma,\lambda,\lambda^{-1})}$ together with an element $\varepsilon\in \chi^{-1}(-1)$ gives us an object in $S'$ by gluing $\oplus_i L_i$ at $p_{n+1}$ and $p_{n+2}$ by the isomorphism 
\beq\label{e25} \oplus_i L_i|_{p_{n+1}}\mapright{\psi_{n+1}} \CC^N \mapright{\varepsilon}\CC^N\mapright{\psi_{n+2}^{-1}} \oplus_i L_i|_{p_{n+2}}.\eeq
We thus have a morphism $\kappa$ above which is \'etale because the fibers parameterize rigidifications at the node. 
The first and last columns come from \eqref{e1} while the left horizontal arrow in the top row is from \eqref{d15}. The middle horizontal arrow in the top row is just the projection. 
The squares with $\Box$ inside are fiber products.  
Gluing a section of $\oplus_i L_i$ by \eqref{e25} at $p_{n+1}$ and $p_{n+2}$ amounts to checking that its image by the upper left vertical arrow $\tilde{\bq}_X$ lies in the image of the top left horizontal arrow in \eqref{e2}, 
since sections of $L_i$ vanish at narrow markings. Hence the fiber product in the top left corner and that in the center give us the same $\hat{X}$.

For cosection localized Gysin maps, we need a bigger diagram involving $Y$'s instead of $X$'s in the second row of \eqref{e2}. To simplify the notation, let 
$$\tilde{Z}=\sqcup_{\lambda\in G}Z_{(\ugamma,\lambda,\lambda^{-1})},\quad
\tilde{S}=\sqcup_{\lambda\in G}S\urig_{g-1,(\ugamma,\lambda,\lambda^{-1})},\quad 
\tilde{X}=\sqcup_{\lambda\in G}X_{g-1,(\ugamma,\lambda,\lambda^{-1})},$$
$$\Delta= (\sqcup_{\lambda\in G} Z_\ugamma\times \CC^{N_\lambda})\times \chi^{-1}(-1),\quad
S=S\urig_{g,\ugamma},\quad
X=X_{g,\ugamma},\quad Z=Z_\ugamma.$$

Let $\pi:\cC\to S$ be a universal curve and $\cL=\oplus_{i=1}^N\cL_i$ be the universal family of spin bundles over $\cC$. As in the construction of \eqref{d34}, we fix an embedding $\cL\to P$ into a locally free sheaf $P$ with $R^1\pi_*P=0$.
Then we have the commutative diagram
\beq\label{e26}\xymatrix{
\tilde{P}\ar@/_1pc/[dd] & \hat{P}\ar[l]\ar[r] &\hat{P}_0\ar[r] & P'\ar[r] &P\ar@/^1pc/[dd] \\
\tilde{\cL}\ar@{^(->}[u]\ar[d]\ar@{}[dr]^{\Box} & \hat{\cL}\ar[r]\ar[l]\ar@{^(->}[u]\ar[d]\ar@{}[dr]^{\Box} &\hat{\cL}_0\ar[r] \ar@{^(->}[u]\ar[d] \ar@{}[dr]^{\Box} & \cL'\ar[r]\ar@{^(->}[u]\ar[d] \ar@{}[dr]^{\Box} &\cL\ar@{^(->}[u]\ar[d] \\
\tilde{\cC}\ar[d]_{\tilde{\pi}} \ar@{}[dr]^{\Box} & \hat{\cC}\ar[r]^c\ar[l]\ar[d]^{\hat{\pi}} &\hat{\cC}_0\ar[r] \ar[d]_{\hat{\pi}_0} \ar@{}[dr]^{\Box} & \cC'\ar[r]\ar[d]^{\pi'} \ar@{}[dr]^{\Box} &\cC\ar[d]^\pi \\
\tilde{S} & \hat{S}\ar[l]_{\tau}\ar@{=}[r] &\hat{S}\ar[r]^\kappa & S'\ar[r] &S
}\eeq
where the bottom row is the third row in \eqref{e2} and the squares with $\Box$ inside are fiber products. The first row is also obtained by fiber product over the third row. The morphism $c$ glues the last two nodes $p_{n+1}$ and $p_{n+2}$. 

For the resolution \eqref{d34}, we consider the exact sequence
$$0\lra \cL \lra P\oplus \sum_{j=1}^n\cL|_{p_j}\lra P\oplus \sum_{j=1}^n\cL|_{p_j}/\cL\lra 0$$
which gives rise to resolutions of $R{\pi}_*{\cL}$ by taking fiber product:
\beq\label{e28} \xymatrix{
[M\ar@{->>}[d]\ar[r]^\alpha & F]\ar@{->>}[d]\\
[\pi_*P\oplus \pi_*\sum_{j=1}^n\cL|_{p_j} \ar[r] & \pi_*(P\oplus \sum_{j=1}^n\cL|_{p_j}/\cL )]}\eeq
after fixing a surjective homomorphism $F \to \pi_*(P\oplus \sum_{j=1}^n\cL|_{p_j}/\cL ) $  from a sufficiently negative vector bundle $F$.
We let $Y=M\times_{B}Z$ with $B=\prod_{j=1}^n\CC^{N_{\gamma_j}}\supset Z=\underline{w}^{-1}(0)$
as in \S\ref{S4.5}.

Likewise, we have resolutions of $R\tilde{\pi}_*\tilde{\cL}$ by locally free sheaves:
\beq\label{e29}
 \xymatrix{
[\tilde{M}\ar[r]^{\tilde{\alpha}}\ar@{->>}[d] &  \tilde{F}]\ar@{->>}[d]\\
[\tilde{\pi}_*\tilde{P}\oplus \tilde{\pi}_*\sum_{j=1}^{n+2}\tilde{\cL}|_{p_j}\ar[r] & \tilde{\pi}_*(\tilde{P}\oplus \sum_{j=1}^{n+2}\tilde{\cL}|_{p_j}/\tilde{\cL})]}\eeq 
for a surjection $\tilde{F}\to \tilde{\pi}_*(\tilde{P}\oplus \sum_{j=1}^{n+2}\tilde{\cL}|_{p_j}/\tilde{\cL})$ with $\tilde{F}$ sufficiently negative.
Let $\tilde{Y}=\tilde{M}\times_{\tilde{B}}\tilde{Z}$ where $\tilde{B}=B\times \CC^{N_\lambda}\times \CC^{N_{\lambda^{-1}}}$ and $$\tilde{Z}=(w_{\gamma_1}\boxplus\cdots \boxplus w_{\gamma_n}\boxplus w_{\lambda}\boxplus w_{\lambda^{-1}})^{-1}(0).$$

Pulling back 
the right vertical in \eqref{e29} to $\hat{S}$,
we have a fiber product  
\beq\label{e31} \xymatrix{
\hat{M}\ar@{->>}[d]\ar[r]^{\hat{\alpha}} & \hat{F}\ar@{->>}[d]\\
\hat{\pi}_*(\hat{P}\oplus \sum_{j=1}^n\hat{\cL}|_{p_j})\oplus (\hat{\pi}_0) _*\hat{\cL}_0|_{p_0}\ar[r]& \hat{\pi}_*(\hat{P}\oplus \sum_{j=1}^{n+2}\hat{\cL}|_{p_j}/\hat{\cL})}\eeq 
where $p_0=c(p_{n+1})=c(p_{n+2})$ denotes the glued node in $\hat{\cC}_0$.

We also have a commutative diagram of short exact sequences 
\beq\label{f50}\xymatrix{
0\ar[d] & 0\ar[d]\\ 
(\hat{\pi}_0)_*(\hat{P}_0\oplus \sum_{j=1}^n\hat{\cL}_0|_{p_j})   \ar[r]\ar[d] &   (\hat{\pi}_0)_*(\hat{P}_0\oplus \sum_{j=1}^{n}\hat{\cL}_0|_{p_j}/\hat{\cL}_0)  \ar[d]\\
\hat{\pi}_*(\hat{P}\oplus \sum_{j=1}^n\hat{\cL}|_{p_j})\oplus (\hat{\pi}_0) _*\hat{\cL}_0|_{p_0} \ar[r]\ar[d]&   \hat{\pi}_*(\hat{P}\oplus \sum_{j=1}^{n+2}\hat{\cL}|_{p_j}/\hat{\cL}) \ar[d]\\
(\hat{\pi}_0)_*(\hat{P}_0|_{p_0}\oplus \hat{\cL}_0|_{p_0})\ar[d]\ar[r]^\cong & (\hat{\pi}_0)_*(\hat{P}_0|_{p_0}\oplus \hat{\cL}_0|_{p_0})  \ar[d]\\
  0   & 0.
}\eeq

Pulling back the right vertical in \eqref{e31} by the top right vertical arrow in \eqref{f50} gives us a fiber product
\beq\label{e32}\xymatrix{
\hat{M}_0\ar@{->>}[d]\ar[r]^{\hat{\alpha}_0}& \hat{F}_0\ar@{->>}[d]\\
(\hat{\pi}_0)_*(\hat{P}_0\oplus \sum_{j=1}^n\hat{\cL}_0|_{p_j})\ar[r] & (\hat{\pi}_0)_*(\hat{P}_0\oplus \sum_{j=1}^{n}\hat{\cL}_0|_{p_j}/\hat{\cL}_0)}\eeq
which gives us two resolutions of $R(\hat{\pi}_0)_*\hat{\cL}_0$. 
Let $\hat{Y}\subset \hat{M}$ be defined by $\hat{Y}=\tilde{Y}\times_{\tilde{Z}}\Delta$.
Let $\hat{Y}_0\subset \hat{M}_0$ be the inverse image of $Z$ by the smooth morphism 
$$\hat{M}_0\lra (\hat{\pi}_0)_*(\hat{P}_0\oplus \sum_{j=1}^n\hat{\cL}_0|_{p_j})\lra (\hat{\pi}_0)_*\sum_{j=1}^n\hat{\cL}_0|_{p_j}\lra B_\ugamma.$$

Now \eqref{f50} induces a commutative diagram of exact sequences
\beq\label{e30}\xymatrix{
0\ar[r] & \hat{M}_0\ar[r]\ar[d]_{\hat{\alpha}_0} & \hat{M}\ar[r]\ar[d]^{\hat{\alpha}} & (\hat{\pi}_0)_*(\hat{P}_0|_{p_0}\oplus \hat{\cL}_0|_{p_0})\ar[r]\ar[d]^\cong & 0\\
0\ar[r] & \hat{F}_0\ar[r] & \hat{F}\ar[r] & (\hat{\pi}_0)_*(\hat{P}_0|_{p_0}\oplus \hat{\cL}_0|_{p_0})\ar[r] &0.
}\eeq



By taking the fiber product, we have the following commutative diagrams
\beq\label{e27}\xymatrix{
\tilde{Z}\ar@{}[dr]^{\Box} & \Delta\ar[l]\ar[r] & Z\\ 
\tilde{Y}\ar[u]^{{\bq}_{\tilde Y}}\ar[d]_{\bp_{\tilde{Y}}} &\hat{Y}\ar[l]_{\tau_Y}\ar[u]^{\bq_{\hat{Y}}}\ar[d] &\hat{Y}_0\ar[l]\ar[d]^{\kappa_Y}\ar[u]^{\bq_{\hat{Y}_0}}  \\
\tilde{S}\ar[d]_{\tilde{\mathrm{st}}} & \hat{S}\ar[l]_\tau\ar@{=}[r] \ar[d]^{\hat{\mathrm{st}}} & \hat{S}\ar[d]^{\hat{\mathrm{st}}}\\
\overline{M}_{g-1,n+2}\ar@{=}[r]&\overline{M}_{g-1,n+2}\ar@{=}[r]&\overline{M}_{g-1,n+2},}
\eeq
\beq\label{f52}
\xymatrix{
Z\ar@{=}[r] & Z\ar@{=}[r] & Z\\
\hat{Y}'_0\ar@{}[dr]^{\Box}\ar[d]\ar[r]^{\kappa_Y} \ar[u]_{\bq_{\hat{Y}'_0}} & Y'\ar@{}[dr]^{\Box}\ar[r]^{u_Y} \ar[u]_{\bq_{Y'}}\ar[d] & Y\ar[d]^{\bp_Y}\ar[u]_{\bq_Y}\\
\hat{S}\ar[r]^\kappa\ar[d]^{\hat{\mathrm{st}}} & S'\ar@{}[dr]^{\Box}\ar[r]\ar[d]^{\mathrm{st}'} & S\ar[d]^{\mathrm{st}}\\
\overline{M}_{g-1,n+2}\ar@{=}[r]&\overline{M}_{g-1,n+2}\ar[r]^u & \rmgn.
}\eeq
The squares with $\Box$ inside are fiber products. The bundles $E, \tilde{E}, \hat{E}, \hat{E}_0, \hat{E}'_0$ on $Y, \tilde{Y}, \hat{Y}, \hat{Y}_0, \hat{Y}'_0$ respectively, and  the sections $s, \tilde{s}, \hat{s}, \hat{s}_0,\hat{s}'_0$, as well as the cosections $\sigma, \tilde{\sigma}, \hat{\sigma}, \hat{\sigma}_0,
\hat{\sigma}'_0$ are defined by the recipe in \S\ref{S4.1} and pulling back. 

Since $\bq_Y$ and $\bq_{\tilde{Y}}$ are smooth on each fiber of $\bp_Y$ and $\bp_{\tilde{Y}}$, all the upward arrows in \eqref{e27} and \eqref{f52} are smooth. 
By Proposition \ref{e35}, $(\hat{s}_0)^!_{\hat{\sigma}_0}\circ \bq_{\hat{Y}_0}$ is independent of the choice of a resolution of 
$R(\hat{\pi}_0)_*\hat{\cL}_0$ (cf.  \S\ref{S4.4} for more detail). 
Because \eqref{e32} and the pullback of \eqref{e31} to $\hat{S}_0$ are just two resolutions of $R(\hat{\pi}_0)_*\hat{\cL}_0$, 
we have the equality 
\beq\label{f53}
(\hat{s}_0)^!_{\hat{\sigma}_0}\circ \bq_{\hat{Y}_0}^*=(\hat{s}'_0)^!_{\hat{\sigma}'_0}\circ \bq_{\hat{Y}'_0}^*.
\eeq


Now we can prove (2). 
By the construction in \S\ref{S4.5}, it is straightforward that $\bq_Y$ and $\bq_{Y'}$ are both smooth and hence $u_Y$ is placid. As $\overline{M}_{g-1,n+2}$ is a smooth divisor in $\rmgn$, $Y'$ is a Cartier divisor of $Y$.
Applying Proposition \ref{b60} and \eqref{c30} to the two columns on the right in \eqref{f52}, we have
\beq\label{e3}
u^!\Omega_{g,n}(v_1,\cdots,v_n)=u^!\frac{(-1)^D}{\deg \mathrm{st}}\mathrm{st}_*s_\sigma^!\bq_Y^*(v_1\otimes\cdots\otimes v_n)
\eeq
\[ = \frac{(-1)^D}{\deg \mathrm{st}'}\mathrm{st}'_*u^!s_\sigma^!\bq_Y^*(v_1\otimes\cdots\otimes v_n)\]
\[= \frac{(-1)^D}{\deg \mathrm{st}'}\mathrm{st}'_*{s'}_{\sigma'}^!u_Y^*\bq_Y^*(v_1\otimes\cdots\otimes v_n) \]
\[ = \frac{(-1)^D}{\deg \mathrm{st}'}\mathrm{st}'_*{s'}_{\sigma'}^!\bq_{Y'}^*(v_1\otimes\cdots\otimes v_n).\]
Since $\kappa$ is \'etale, $$\frac{(-1)^D}{\deg \hat{\mathrm{st}}}\hat{\mathrm{st}}_*\circ\kappa^!=\frac{(-1)^D}{\deg {\mathrm{st}'}}\mathrm{st}'_*.$$ 
By Proposition \ref{c17}, $\kappa^!\circ {s'}^!_{\sigma'}=(\hat{s}'_0)_{\hat{\sigma}'_0}^!\circ \kappa_Y^*$.
 By \eqref{f53}, \eqref{e3} thus equals
\beq\label{e33} 
\frac{(-1)^D}{\deg \hat{\mathrm{st}}}\hat{\mathrm{st}}_*(\hat{s}'_0)_{\hat{\sigma}'_0}^!{\bq}_{\hat{Y}'_0}^*(v_1\otimes\cdots\otimes v_n)=\frac{(-1)^D}{\deg \hat{\mathrm{st}}}\hat{\mathrm{st}}_*(\hat{s}_0)_{\hat{\sigma}_0}^!{\bq}_{\hat{Y}_0}^*(v_1\otimes\cdots\otimes v_n).
\eeq
By Proposition \ref{e34} and \eqref{e30}, we then have 
\[ (\hat{s}_0)_{\hat{\sigma}_0}^!\bq^*_{\hat{Y}_0}(v_1\otimes\cdots\otimes v_n)
={\hat{s}}_{\hat{\sigma}}^!{\bq}_{\hat{Y}}^*(\sum_{\lambda\in G}v_1\otimes\cdots\otimes v_n\otimes [\CC^{N_\lambda}]).
\]
Hence, \eqref{e33} equals 
\beq\label{e4}
\frac{(-1)^D}{\deg \hat{\mathrm{st}}}\hat{\mathrm{st}}_*{\hat{s}}_{\hat{\sigma}}^!{\bq}_{\hat{Y}}^*(\sum_{\lambda\in G}v_1\otimes\cdots\otimes v_n\otimes [\CC^{N_\lambda}]).
\eeq

Since $\tau$ is \'etale of degree $|G|$, 
$$\frac{(-1)^D}{\deg \tilde{\mathrm{st}}}\tilde{\mathrm{st}}_*\circ \frac1{|G|}\tau_*=\frac{(-1)^D}{\deg \hat{\mathrm{st}}}\hat{\mathrm{st}}_*.$$
This implies that \eqref{e4} equals
\beq\label{e5}
\frac{(-1)^D}{\deg \tilde{\mathrm{st}}}\tilde{\mathrm{st}}_*\circ \frac1{|G|}\tau_*{\hat{s}}_{\hat{\sigma}}^!\bq^*_{\hat{Y}}(v_1\otimes\cdots\otimes v_n\otimes [\CC^{N_\lambda}]).
\eeq
By Proposition \ref{c9}  together with \eqref{d21}, \eqref{e5} equals
\beq\label{e6}
\frac{(-1)^D}{\deg \tilde{\mathrm{st}}}\tilde{\mathrm{st}}_*\circ \frac1{|G|}{\tilde{s}}_{\tilde{\sigma}}^!\bq^*_{\tilde{Y}}\left(v_1\otimes\cdots\otimes v_n\otimes |G|\cdot(\sum_{k,l}c^{kl}e_k\otimes e_l)\right)
\eeq
\[ = \Omega_{g-1,n+2}\left(v_1\otimes\cdots\otimes v_n\otimes (\sum_{k,l}c^{kl}e_k\otimes e_l)\right).\]
This proves the splitting principle (2).

For (3), we consider the gluing morphism
$$u:\overline{M}_{g_1,n_1+1}\times \overline{M}_{g_2,n_2+1}\lra \rmgn$$
of the last nodes $p_+$ and $p_-$ of curves $(C_1,p_1,\cdots,p_{n_1},p_+)\sqcup (C_2,p_{n_1+1},\cdots,p_{n},p_-)$ in $\overline{M}_{g_1,n_1+1}\times \overline{M}_{g_2,n_2+1}$ with $g=g_1+g_2$ and $n=n_1+n_2$ and  
the commutative diagram
\footnotesize
\beq\label{e7}\xymatrix{
\sqcup_{\lambda\in G}Z_{(\ugamma_1,\lambda)}\times Z_{(\ugamma_2,\lambda^{-1})}\ar@{}[dr]^{\Box} & \hat{Z}_\ugamma \ar[l]\ar[r] & Z_\ugamma\ar@{=}[r] & Z_\ugamma\\
\sqcup_{\lambda\in G}X_{g_1,(\ugamma_1,\lambda)}\times X_{g_2,(\ugamma_2,\lambda^{-1})}\ar[u]_{\tilde{\bq}}\ar[d] & \hat{X}\ar@{}[dr]^{\Box}\ar[l]\ar[r]\ar[u]_{\hat{\bq}} \ar[d] & X'\ar@{}[dr]^{\Box}\ar[r]^{u_X} \ar[u]_{\bq_{X'}}\ar[d] & X_{g,\ugamma}\ar[d]^{\bp_X}\ar[u]_{\bq_X}\\
\sqcup_{\lambda\in G}S\urig_{g_1,(\ugamma_1,\lambda)}\times S\urig_{g_2,(\ugamma_2,\lambda^{-1})}\ar[d]_{\tilde{\mathrm{st}}} & \hat{S}\ar[l]_\tau\ar[r]^\kappa \ar[d]^{\hat{\mathrm{st}}} & S'\ar@{}[dr]^{\Box}\ar[r]\ar[d]^{\mathrm{st}'} & S\urig_{g,\ugamma}\ar[d]^{\mathrm{st}}\\
\overline{M}_{(g_1,g_2),(n_1,n_2)}\ar@{=}[r]&\overline{M}_{(g_1,g_2),(n_1,n_2)}\ar@{=}[r]&\overline{M}_{(g_1,g_2),(n_1,n_2)}\ar[r] & \rmgn.
}\eeq
\normalsize
Here $\hat{Z}_\ugamma=(\sqcup_{\lambda\in G} Z_\ugamma\times \CC^{N_\lambda})\times \chi^{-1}(-1)$, $\overline{M}_{(g_1,g_2),(n_1,n_2)}=\overline{M}_{g_1,n_1+1}\times \overline{M}_{g_2,n_2+1}$ and the squares with $\Box$ inside are fiber products. The morphism $\tau$ is  the additional choice of gluing isomorphism  in $\chi^{-1}(-1)\subset (\CC^*)^N$ of the fibers of the line bundles $\{L_i\}_{1\le i\le N}$ over the last two nodes. The morphism $\kappa$ is the additional choice of the rigidification of the fibers over the node.
Now the proof of (3) is almost the same as that for (2). 

For (4), we consider the diagram
\beq\label{e8}\xymatrix{
Z_\ugamma\ar@{=}[r] & Z_\ugamma\ar@{=}[rr] && Z_\ugamma \\
X_{g,(\ugamma,\mathbf{1})}\ar@{}[ddr]^{\Box}\ar[r] \ar[u]_{\hat{\bq}} \ar[dd] & X'\ar@{}[drrr]^{\Box}\ar[rr] \ar[u]_{\bq'} \ar[dd]\ar@{^(->}[dr] && X_{g,\ugamma}\ar[u]_{\bq} \ar[dd]\ar@{^(->}[dr] \\
&&M'\ar[rr]\ar[dl]\ar@{}[dr]^{\Box} && M\ar[dl] \\
S\urig_{g,(\ugamma,\mathbf{1})}\ar[r]^\kappa \ar[d]_{\hat{\mathrm{st}}} & S'\ar@{}[drr]^{\Box}\ar[d]_{\mathrm{st}'}\ar[rr] && S\urig_{g,\ugamma}\ar[d]^{\mathrm{st}}\\
\overline{M}_{g,n+1}\ar@{=}[r] & \overline{M}_{g,n+1}\ar[rr]^\theta && \rmgn
}\eeq
where the squares with $\Box$ are fiber products and $\kappa$ is the forgetting of the rigidification at the last marking which is \'etale. Here $\bq, \bq', \hat{\bq}$ are smooth by construction.
Now (4) follows from Proposition \ref{c17} by a similar computation like (2) above. 

\def\lloc{_{\mathrm{loc}}}
When all the markings are narrow so that $N_{\gamma_j}=0$ for all $1\le j\le n$, $B_\ugamma=Z_\ugamma=0$ and $\sigma=0$ with $Y=M$ in the setup of \S\ref{S4.1}. Since $[Y]$ is the image of the fundamental cycle of $Y$ by the cycle class map $h_Y$, by \eqref{c8}, we have  
$$s^!_\sigma[Y]=[Y]\cap e_\sigma(E,s)=h_S[X]\virt\lloc$$
where $[X]\virt\lloc$ denotes the cosection localized virtual cycle defined in \cite{CLL}.  
By \cite[Theorem 5.6]{CLL}, $(-1)^D[X]\virt\lloc$ coincides with the FJRW class for narrow sectors. 
Hence the cohomological field theory defined by \eqref{d40} coincides with that in \cite{FJR} for narrow sectors.
\end{proof}

There is another way to phrase our cohomological field theory \eqref{d40}. By using the isomorphism $\cH^\vee\cong \cH$ by the perfect pairing on $\cH$, we may think of the homomorphism \eqref{d35} as a class
\[  [X\urig_{g,\ugamma}]\virt\in H_*(S\urig_{g,\ugamma})\otimes \prod_{j=1}^n\cH_{\gamma_j} \]
of degree $6g-6+2n+2\sum_i\chi(L_i).$ Letting $\mathrm{so}:S\urig_{g,\ugamma}\to S_{g,\ugamma}$ denote the map forgetting the rigidification, we define 
\beq\label{e20} [X_{g,\ugamma}]\virt =\frac{(-1)^D}{\deg \mathrm{so}}\mathrm{so}_*[X\urig_{g,\ugamma}]\virt\eeq
where $D=-\sum_i\chi(L_i)$. 
\begin{theo}\label{e21}
The class $[X_{g,\ugamma}]\virt$ in \eqref{e20} satisfies all the axioms in \cite[Theorem 4.1.8]{FJR}.
\end{theo}
\begin{proof}
Axioms (1)-(4) and (10) are easy and left to the reader. Axioms (5)-(7) are essentially the axioms for the cohomological field theory that we proved above. (8) follows from our construction and the Thom-Sebastiani isomorphism. (9) holds because in fact our cosection localized Gysin map is defined over $\QQ$.  
\end{proof}

\begin{rema}
We used the decomposition theorem of \cite{BBD} for the cosection localized Gysin map in Theorem \ref{b18}.
Instead we could also use Theorem \ref{f18} and avoid the decomposition theorem. 
\end{rema}

\subsection{Independence of choices}\label{S4.4}

The only remaining choice for our construction of quantum singularity theories is the resolution $[M\mapright{\alpha} F]$ of $R\pi_*(\oplus \cL_i)$ in \eqref{d34}.\footnote{Actually there is one more choice, namely that of the cosection $\sigma_M$
in a homotopy class \cite[\S4.2]{PV}. But two such choices are connected by a continuous path and the homomorphism $s^!_\sigma$, which is defined over $\QQ$,  is independent of the choice of the cosection in the homotopy class.} 
In this subsection, we check that our cohomological field theory is independent of the choice of the resolution of $R\pi_*(\oplus_i\cL_i)$.

If two resolutions $[M\mapright{\alpha} F]$ and $[M'\mapright{\alpha'} F']$ (with smooth $M\to B$ fiberwisely) by locally free sheaves are isomorphic in the derived category, it is easy to see that there is a third complex $[M''\to F'']$ of locally free sheaves dominating the two quasi-isomorphically by chain maps.
So we only need to consider the case where there is a chain map $f=(f_0,f_1):[M'\to F']\to [M\to F]$ which we may assume surjective by replacing $M'$ and $F'$ by $M'\oplus M$ and $F'\oplus M$.
Indeed we have quasi-isomorphisms $[M'\to F']\Leftrightarrow [M'\oplus M\to F'\oplus M]$ and a surjective quasi-isomorphic chain map $[M'\oplus M\to F'\oplus M]\to [M\to F]$ 
where the arrow $M'\oplus M\to F'\oplus M$ is $(a',a)\mapsto (\alpha'(a'),a)$,
the arrow $M'\oplus M\to M$ is $(a',a)\mapsto f_0(a')+a$ and the arrow $F'\oplus M\to F$ is $(b',a)\mapsto f_1(b')+\alpha(a)$.

Let $[K\to K]$ denote the kernel of the surjective chain map $f$ and let $s$ and $s'$ denote the section of a vector bundle $E$ over $Y$ and $E'$ over $Y'$ respectively by the recipe in \S\ref{S4.1}.
Then by Proposition \ref{e35}, $s^!_\sigma(\xi)=(s')^!_{\sigma'}f_0^{*}\xi$ for $\xi\in \ih_*(Y)$. Thus we obtain the independence.

\bigskip

\section{Gauged linear sigma models}\label{S5}

The quantum singularity theories discussed above can be generalized to the setting where $G$ is not necessarily finite but a reductive group. 

\subsection{The input data}
Let $V=\CC^N$ and let $\hat{G}$ be a subgroup of $GL(V)$ satisfying the following:
\begin{enumerate}
\item $\hat{G}$ contains the torus $T=\{\mathrm{diag}(t^{d_1},\cdots, t^{d_N})\}\cong \CC^*$, for nonnegative integers $d_1,\cdots,d_N$ whose greatest common divisor is $1$;
\item there is a surjective homomorphism $\chi:\hat{G}\to \CC^*$ whose restriction to $T$ is $\mathrm{diag}(t^{d_1},\cdots, t^{d_N})\mapsto t^d$ for some $d>0$.
\end{enumerate}

Let $G=\mathrm{ker}(\chi)$; it fits into the short exact sequence
$$1\lra G\lra \hat{G}\mapright{\chi} \CC^*\lra 1.$$
We assume that the elements of $G$ and $T$ commute in $\hat{G}$, i.e. $T$ belongs to the  centralizer $C(G)=\{t\in \hat{G}\,|\, tg=gt,\ \forall g\in G\}$ of $G$. 
Then the inclusion maps of $G$ and $T$ in $\hat{G}$ induce an isomorphism $G\times T/\mu_d\cong \hat{G}$ where $\mu_d$ is the cyclic group of $d$-th roots of unity in $\CC$. 

Let $w:V=\CC^N\to \CC$ be a polynomial satisfying the following:
\begin{enumerate}
\item[(i)] $w$ is nondegenerate, i.e. the critical locus $[\mathrm{Crit}(w)/G]$ in $[V/G]$ is proper where $\mathrm{Crit}(w)=\mathrm{zero}(dw)$;
\item[(ii)] $w$ is $(\hat{G},\chi)$-semi-invariant, i.e. $w(g\cdot z)=\chi(g)w(z)$ for $g\in \hat{G}$ and $z\in V$.  
\end{enumerate}

The condition (ii) above implies that 
\begin{enumerate}
\item[(a)] $w$ is quasi-homogeneous: $w(t^{d_1}z_1,\cdots, t^{d_N}z_N)=t^dw(z_1,\cdots,z_N)$;
\item[(b)] $w$ is $G$-invariant: $w(g\cdot z)=w(z)$ for $g\in G$. 
\end{enumerate}

\begin{exam}\label{f30}
Let $V=\CC^6$ and $\hat{G}=\{(t,\cdots,t,t_0)\}\cong (\CC^*)^2$. Let $\chi:\hat{G}\to \CC^*$ be
$\chi(t,t_0)=t^5t_0$. Then $G=\ker(\chi)=\{(t,\cdots, t, t^{-5})\}\cong \CC^*$. 
Let $w(z_1,\cdots,z_5,z_0)=z_0\sum_{i=1}^5z_i^5$. Then $w((t,t_0)\cdot z)=\chi(t,t_0)w(z)$ and 
$$\mathrm{Crit}(w)=\{(z_1,\cdots,z_5,0)\,|\,\sum_{i=1}^5 z_i^5=0\}\cup \{(0,\cdots,0,z_0)\,|\,z_0\in \CC\}.$$
Hence $[\mathrm{Crit}(w)/G]$ is proper. 
\end{exam}

Since $G$ is not necessarily finite, to construct a GIT quotient of $V$ by $\hat{G}$  
we pick a character $\hat{\theta}:\hat{G}\to \CC^*$. Let 
$$V^{ss}_{\hat{G},\hat{\theta}}=
\Bigl\{z\in V\,\Big|\ \parbox{22em}{for some $a>0$, there are 
$(\hat{G},\hat{\theta}^a)$-semi-invariant nonconstant polynomial $f$ such that $f(z)\ne 0$}\,\Bigr\}.$$
We call it the open set of $(\hat{G},\hat{\theta})$-semistable points in $V$. 
Let $\theta=\hat{\theta}|_G$. Likewise, $V^{ss}_{G,\theta}$ is defined similarly. We assume that
$$V^{ss}_{\hat{G},\hat{\theta}}=V^{ss}_{G,\theta}=:V^{ss}_\theta$$
and that there are no strictly semistable points. The GIT quotient of $V$ by $G$ with respect to $\theta$ is the quotient
$[V^{ss}_\theta/G]$ and the restriction of $w$ to $V^{ss}_\theta$ induces a regular function $w_\theta$ on the quotient. 

\begin{exam}\label{f31} (Continuation of Example \ref{f30}.)
Let $\hat\theta_+$ (resp. $\hat\theta_-$) be the homomorphism $\hat{G}\to \CC^*$ defined by $(t,t_0)\mapsto t$ (resp. $(t,t_0)\mapsto t_0$). Then $V_{\theta_+}^{ss}=(\CC^5-0)\times \CC$ and $V\git_{\theta_+}G\cong \sO_{\PP^4}(-5)=K_{\PP^4}$.
Tensoring the section $\sum_{i=1}^5z_i^5\in H^0(\sO_{\PP^4}(5))$ gives us the regular function $w_+$ on $V\git_{\theta_+}G$ induced by $w$.  
Similarly, $V_{\theta_-}^{ss}=\CC^5\times \CC^*$ and $V\git_{\theta_-}G\cong \CC^5/\mu_5$ on
which $w$ restricts to $w_-=\sum_{i=1}^5z_i^5$.
\end{exam}

The input datum for a GLSM is the triple $(w, G, \theta)$ for $w$, $G$ and $\theta$ as above. 

\subsection{The state spaces}

Fix the input datum $(w, G, \theta)$ for a GLSM. For $\gamma \in G$, let 
$$V_{\gamma,\theta}=[\{(z,g)\in V^{ss}_\theta\times G\,|\, g\in \langle\gamma\rangle, gz=z \}/G]
=[(V_\theta^{ss})^\gamma/C(\gamma)]$$
where $(V_\theta^{ss})^\gamma$ is the $\gamma$-fixed locus in $V_\theta^{ss}$,  $\langle \gamma\rangle$ denotes the conjugacy class of $\gamma$ and $C(\gamma)=\{ g\in G\,|\,g\gamma=\gamma g\}.$
Let $w_\gamma^\infty\subset V_{\gamma,\theta}$ denote the locus 
where $\mathrm{Re}\, w(z)>\!>0$. 
For a $G$-invariant closed $V'$ in $V$, $V'_{\gamma,\theta}$ is defined in the same way with $V$ replaced by $V'$. 

The state space for the GLSM is 
$$\cH=\bigoplus_{\langle\gamma\rangle \subset G}\cH_\gamma, \quad \cH_\gamma=\bigoplus_{\alpha\in \QQ} H^{\alpha-2\mathrm{age}(\gamma)+2q}(V_{\gamma,\theta},w_\gamma^\infty) 
$$
where $\langle \gamma\rangle$ runs through the conjugacy classes in $G$ and $q=\frac1{d}\sum_id_i$. 

\begin{exam}\label{f32} (Continuation of Example \ref{f31}.)
For $\theta_+$, $(V_{\theta_+}^{ss})^\gamma=\emptyset$ for $\gamma\ne 1$ and  $K_{\PP^4}$ for $\gamma=1$.
As $\mathrm{age}(1)=0$ and $q=1$, we have
$$\cH_+^i= H^{i+2}(K_{\PP^4},w_+^\infty)\cong H^i(Q_w)\cong H_{i+2}(K_{\PP^4}|_{Q_w})$$
where $Q_w\subset \PP^4$ is the Fermat quintic. Thus $\cH_+\cong H^*(Q_w)$ as graded vector spaces. 
As $w_+^{-1}(0)=\PP^4\cup K_{\PP^4}|_{Q_w}$, we have a homomorphism 
$$H_*(K_{\PP^4}|_{Q_w})\to H_*(w_+^{-1}(0))$$
which is injective by direct computation. 

For $\theta_-$, $(V_{\theta_-}^{ss})^\gamma=\emptyset$ for $\gamma\notin \mu_5$. 
When $\gamma=1$, $$\cH_-^i=H^5(\CC^5,w^\infty)^{\mu_5}\cong H^3(Q_w)
$$ for $i=3$ and $0$ otherwise.
When $\gamma=e^{2\pi\sqrt{-1}a/5}$ for $1\le a\le 4$, $\cH_-^i=\CC$ for $i=2a-2$ and $0$ otherwise. 
Summing these for $\gamma\in \mu_5$,
$\cH_-\cong H^*(Q_w)$ as graded vector spaces. 
From \S\ref{S2.5}, $H^3(Q_w)\cong H_5(w_-^{-1}(0))$. 
\end{exam}

For an $n$-tuple $\underline{\gamma}=(\gamma_1,\cdots,\gamma_n)\in G^n$, let 
$\underline{w}$ denote the Thom-Sebastiani sum of $w_{\gamma_j}$ for $j=1,\cdots, n$ where $w_{\gamma_j}$
is the restriction of $w$ to the $\gamma_j$-fixed locus.  
By the Thom-Sebastiani map \eqref{b17}, we have a homomorphism
\beq\label{f37} H_*(w_{\gamma_1}^{-1}(0))\otimes \cdots \otimes H_*(w_{\gamma_n}^{-1}(0))\lra H_*(\underline{w}^{-1}(0)).\eeq



\subsection{The moduli spaces}

Let $R=R_{g,\underline{\gamma}}$ denote the moduli stack of quadruples $(C,p_j, P, \varphi)$ where $(C,p_j)$ is a prestable twisted curve of genus $g$ with $n$ markings, $P$ is a principal $\hat{G}$-bundle of type $\underline{\gamma}$, and $\varphi$ is an isomorphism
$\chi_*(P)\cong P(\omega_C\ulog)$ as in \S\ref{S4.5}. Here the type $\underline{\gamma}\in G^n$ is defined as in \S\ref{S4.5}.

Then by \cite[\S4.2]{FJRgl}, for any $G$-invariant closed $V'\subset V$, we have a separated \DM stack of finite type  
$$LGQ^{\varepsilon,\theta}_{g,n}(V',\mathbf{d})=\bigsqcup_{\ugamma} LGQ^{\varepsilon,\theta}_{g,\ugamma}(V',\mathbf{d})$$ which represents the moduli functor
of quintuples $(C,p_j, P, \varphi, x)$ with the quadruple $(C,p_j, P, \varphi)$ lying in $R$
and $x\in H^0(P\times_{\hat{G}}V)$ of fixed degree $\mathbf{d}$ such that
$(C,p_j,x)$ is an $\varepsilon$-stable quasi-map into $V'$ for general $\varepsilon>0$. 
Here the $\varepsilon$-stability means that 
the length of the base locus of $x$ at any point is at most $1/\varepsilon$ and $\omega_C\ulog\otimes L^\varepsilon_{\hat{\theta}}$ is ample where $L_{\hat{\theta}}=P\times_{\hat{G}}\CC(\hat{\theta})$. 
If the GIT quotient  $V'/\!/_{\theta}G$ is proper, so is $LGQ^{\varepsilon,\theta}_{g,n}(V',\mathbf{d})$. 
In particular, as $[\mathrm{Crit}(w)/G]$ is proper by assumption, 
$LGQ^{\varepsilon,\theta}_{g,n}(\mathrm{Crit}(w),\mathbf{d})$ is a proper separated \DM stack. 

\begin{exam}\label{f33} (Continuation of Example \ref{f32}.)
A principal $\hat{G}$-bundle $P$ gives us two line bundles $L$ and $L_0$. An isomorphism of $\chi_*(P)$ with $P(\omega_C\ulog)$ is an isomorphism
$\varphi:L^5L_0\to \omega_C\ulog,$ Hence $L_0\cong L^{-5}\omega_C\ulog$. 
As $P\times_{\hat{G}}V=L^{\oplus 5}\oplus L_0$, $x\in H^0(L^{\oplus 5}\oplus L_0)$ consists of
5 sections $x_1,\cdots, x_5\in H^0(L)$ and a section $x_0\in  H^0(L^{-5}\omega_C\ulog)$. 

For $\theta_+$ and $\varepsilon=\infty$, $(x_1,\cdots, x_5)$ defines a stable map $C\to \PP^4$ and $x_0$ is the p-field 
(cf. \cite{CLp}). 
For $\theta_-$ and $\varepsilon=\infty$, $x_0:L^5\to \omega_C\ulog$ is an isomorphism and hence defines 
a $\mu_5$-spin structure on $C$. This is the case we studied in \S\ref{S4}.
\end{exam}

\subsection{A construction of GLSM invariants}
To construct GLSM invariants by cosection localization, we assume the following.
\begin{assu}\label{f35}
Suppose our GLSM given by $(w,G,\theta)$ satisfies the following.
\begin{enumerate}
\item For $\gamma\in G$, the state space $\cH_\gamma$ is a subspace of $H_*(w_{\gamma,\theta}^{-1}(0))$ where $w_{\gamma,\theta}:V_{\gamma,\theta}\to \CC$ is the regular function induced from $w$.
\item Let $X=LGQ^{\varepsilon,\theta}_{g,n}(V,\mathbf{d})$ and $S=LGQ^{\varepsilon,\theta}_{g,n}(\mathrm{Crit}(w),\mathbf{d})$. There is a smooth \DM stack $W$ and a vector bundle $E_W$ on $W$ such that $X=s_W^{-1}(0)$ for a section $s_W$ of $E_W$. Moreover there is a cosection $\sigma_W:E_W\to \sO_W$ satisfying $X\cap \sigma_W^{-1}(0)=S$. 
\item The evaluation map $q:W\to \prod_{j=1}^n V_{\gamma_j,\theta}$ is a smooth morphism. 
\item Let $\underline{w}$ be the Thom-Sebastiani sum of $w_{\gamma_j,\theta}$ for $1\le j\le n$. Then $\underline{w}\circ q=\sigma_W\circ s_W$.  
\end{enumerate}
\end{assu}

Under Assumption \ref{f35}, we can construct GLSM invariants by cosection localization (Theorem \ref{f18}) as follows. 
By item (1) and \eqref{f37}, we have a homomorphism 
\beq\label{f38} \cH_{\gamma_1}\otimes\cdots\otimes \cH_{\gamma_n}\lra H_*(\underline{w}^{-1}(0)).\eeq
By items (3) and (4), we have the pullback homomorphism 
\beq\label{f39} H_*(\underline{w}^{-1}(0)) \lra H_*((\underline{w}\circ q)^{-1}(0))=H_*((\sigma_W\circ s_W)^{-1}(0)).\eeq
By item (2), Assumption \ref{f21} holds. Hence by Theorem \ref{f18}, we have the cosection localized Gysin map
\beq\label{f40} s^!_\sigma:H_*((\sigma_W\circ s_W)^{-1}(0))\lra H_{*-2r}(S).\eeq
As $S=LGQ^{\varepsilon,\theta}_{g,n}(\mathrm{Crit}(w),\mathbf{d})$ is proper, the forgetful morphism 
$S\to \rmgn$ gives us the pushforward homomorphism 
\beq\label{f41} H_{*}(S)\lra H_*(\rmgn).\eeq

\begin{defi}\label{f36} Under Assumption \ref{f35}, we define 
$$\Omega_{g,n}:\cH^{\otimes n}\lra H_*(\rmgn)\cong H^*(\rmgn)$$
by composing \eqref{f38}, \eqref{f39}, \eqref{f40} and \eqref{f41} and summing them up for all types $\ugamma=(\gamma_1,\cdots,\gamma_n)$.
\end{defi}

We raise the following natural questions. 
\begin{ques}
(1) Does Assumption \ref{f35} hold for all GLSMs?

(2) Do the homomorphisms $\{\Omega_{g,n}\}$ in Definition \ref{f36} satisfy the axioms of cohomological field theory in \S\ref{S4.3}?
\end{ques}

Note that item (1) in Assumption \ref{f35} holds in the Fermat quintic case 
\beq\label{f57} w:\CC^6\git_{\theta_\pm}\CC^*\to \CC,\quad w(z_1,\cdots,z_5,z_0)=z_0\sum_{i=1}^5z_i^5\eeq
by Example \ref{f32} above. In fact, the proof of Lemma \ref{g0} also proves item (1) for hybrid models.  
For items (2)-(4), Ionut Cioncan-Fontanine, David Favero, Jeremy Guere, Bumsig Kim and Mark Shoemaker prove the following.
\begin{theo}\cite[\S4]{CFGK}
Let $(w,G,\theta)$ be a convex hybrid GLSM model and $\epsilon=\infty$. 
Then items (2)-(4) of Assumption \ref{f35} hold. 
\end{theo}
Using this theorem, the authors of \cite{CFGK} construct fundamental matrix factorizations which produce GLSM invariants
by Fourier-Mukai type transformations and Hochschild homology. 
We refer to \cite{CFGK} for definitions of convex hybrid models.
Let us just mention that the Fermat quintic case \eqref{f57} is a 
convex hybrid model. Therefore, Definition \ref{f36} defines GLSM invariants in this case by cosection localization. 



\bigskip

\bibliographystyle{amsplain}

\end{document}